\documentclass[11pt]{amsart}
\usepackage[colorlinks,allcolors=blue]{hyperref}
\usepackage{amsmath,amscd, verbatim}
\usepackage{amsfonts, enumitem}
\usepackage{amssymb, xcolor}
\usepackage{mathrsfs}
\usepackage{tikz-cd}
\usepackage[margin=1in]{geometry}

\setlist[enumerate,1]{label={\normalfont (\arabic*)}}
\setlist[enumerate,2]{label={\normalfont (\alph*)}}

\DeclareMathAlphabet{\altmathbb}{U}{bbold}{m}{n}
\DeclareFontFamily{U}{mathx}{}
\DeclareFontShape{U}{mathx}{m}{n}{<-> mathx10}{}
\DeclareSymbolFont{mathx}{U}{mathx}{m}{n}
\DeclareMathAccent{\wc}{0}{mathx}{"71}

\newcommand{\QQ}{{\mathbb Q}}
\newcommand{\FF}{{\mathbb F}}
\newcommand{\ZZ}{{\mathbb Z}}
\newcommand{\CC}{{\mathbb C}}

\newcommand{\bG}{{\mathbf G}}
\newcommand{\bK}{{\mathbf K}}
\newcommand{\bL}{{\mathbf L}}
\newcommand{\bP}{{\mathbf P}}
\newcommand{\bT}{{\mathbf T}}
\newcommand{\bU}{{\mathbf U}}
\newcommand{\bX}{{\mathbf X}}
\newcommand{\bY}{{\mathbf Y}}

\newcommand{\cD}{{\mathcal D}}
\newcommand{\cE}{{\mathcal E}}

\newcommand{\rC}{{\mathrm C}}
\newcommand{\rO}{{\mathrm O}}
\newcommand{\rN}{{\mathrm N}}
\newcommand{\rZ}{{\mathrm Z}}

\newcommand{\scT}{\mathscr{T}}

\newcommand{\Ad}{\mathrm{Ad}}
\newcommand{\simc}{\mathrm{sc}}
\newcommand{\der}{\mathrm{der}}
\newcommand{\Tr}{\mathrm{Tr}}

\DeclareMathOperator{\Aut}{Aut}
\DeclareMathOperator{\Hom}{Hom}
\DeclareMathOperator{\res}{Res}
\DeclareMathOperator{\ind}{Ind}
\DeclareMathOperator{\irr}{Irr}
\DeclareMathOperator{\cf}{cf}

\newcommand{\wh}[1]{\widehat{#1}}
\newcommand{\wt}[1]{\widetilde{#1}}
\newcommand{\gal}{\mathcal{G}}
\newcommand{\tw}[1]{{}^#1\!}
\newcommand{\abs}[1]{\lvert #1 \rvert}
\newcommand{\Lang}{\mathscr{L}}
\newcommand{\triv}{\altmathbb{1}}

\newtheorem{theorem}{Theorem}[section]
\newtheorem{lemma}[theorem]{Lemma}
\newtheorem{proposition}[theorem]{Proposition}
\newtheorem{corollary}[theorem]{Corollary}

\theoremstyle{definition}
\newtheorem{definition}[theorem]{Definition}

\numberwithin{equation}{section}
\allowdisplaybreaks

\begin{document}

\bibliographystyle{ieeetr}

\title[Galois Automorphisms and Jordan Decomposition]
{Galois automorphisms and a unique Jordan decomposition in the case of connected centralizer}
\author{A. A. Schaeffer Fry}
\address[A. A. Schaeffer Fry]{Dept. Mathematics - University of Denver, Denver, CO 80210, USA; and 
Dept. Mathematics and Statistics - MSU Denver, Denver, CO 80217, USA}
\email{mandi.schaefferfry@du.edu}

\author{Jay Taylor}
\address[J. Taylor]{Department of Mathematics, The University of Manchester, Oxford Road, Manchester, M13 9PL, UK}
\email{jay.taylor@manchester.ac.uk}

\author{C. Ryan Vinroot}
  \address[C. Ryan Vinroot]{Department of Mathematics\\
           William \& Mary\\
           P. O. Box 8795\\
           Williamsburg, VA  23187-8795, USA}
   \email{vinroot@math.wm.edu}

\begin{abstract}
We show that the Jordan decomposition of characters of finite reductive groups
can be chosen so that if the centralizer of the relevant semisimple element in
the dual group is connected, then the map is Galois-equivariant. Further, in
this situation, we show that there is a unique Jordan decomposition satisfying
conditions analogous to those of Digne--Michel's unique Jordan decomposition in
the connected center case.

2020 {\it AMS Subject Classification:}  20C33
\end{abstract}

\thanks{The authors gratefully acknowledge support from a SQuaRE at the American
Institute of Mathematics. The first-named author also gratefully acknowledges
support from the National Science Foundation, Award No. DMS-2100912, and her
former institution, Metropolitan State University of Denver, which holds the
award and allows her to serve as PI. The third-named author was supported in
part by a grant from the Simons Foundation, Award \#713090. The authors also
thank Gunter Malle for comments on an earlier draft of the manuscript and the
referee, whose comments helped improve the exposition.
%The authors also thank Lena Lagamorph for providing unnecessary furry drama to
%the writing of this manuscript.
}
\maketitle

\section{Introduction}
Given a finite group $G$, the fields of values of the irreducible complex
characters of $G$, $\irr(G)$, have revealed themselves to be valuable and
interesting number-theoretic data corresponding to the structure of $G$.
Numerous examples of results demonstrating this involve the rational-valued
characters of $G$, real-valued characters of $G$, and the question of whether a
representation of $G$ can be realized over the field of values of its
characters.  Thus, the Galois action on these fields of character values becomes
a key problem in the character theory of finite groups.

In this paper, we study the action of
$\gal:=\mathrm{Gal}(\mathbb{{Q}^{\rm{ab}}}/\mathbb{Q})$ on the set $\irr(G)$,
where $G$ is a finite reductive group.  This family of groups is of particular
interest because of their role as subgroups of connected reductive algebraic
groups, their actions on finite geometries, and their relation to finite simple
groups.  In particular, for finite groups of Lie type, a key to understanding
the action of $\gal$ on the set $\irr(G)$ is understanding how various
parametrizations of the set $\irr(G)$ behave under the action of $\gal$.  In
\cite{SrVi15, SrVi19}, Srinivasan and the third-named author study this question
for the Jordan decomposition of characters, in the case that the underlying
algebraic group has a connected center.  In \cite{SF19}, the first-named author
studies this question for the Howlett--Lehrer parametrization of Harish-Chandra
series.  More results have been obtained in \cite{Ge03} for the case of
connected center, and the authors have studied the question of fields of values
of characters in \cite{SFV19, SFT22}.

The question of the action of $\gal$ on $\irr(G)$ is particularly difficult in
the case that the underlying algebraic group has  disconnected center, and will
play a crucial role in, for example, proving the inductive Galois--McKay
conditions of \cite{NSV} to prove the McKay--Navarro conjecture for odd primes.
(Note that for the prime $2$, this was finished in \cite{RSF22, RSF23}.) 

The results of \cite{SrVi15, SrVi19} make essential use of the unique Jordan
decomposition proved by Digne and Michel in \cite{DiMi90} in the case of
connected center. The goal of the present paper is twofold: we extend the
results of \cite{SrVi19} on the action of $\gal$ on Jordan decomposition to the
case that the underlying group does not necessarily have a connected center, but
that the semisimple element in question has a connected centralizer in the dual
group; and we extend the results of \cite{DiMi90} to show that there is a unique
Jordan decomposition satisfying properties analogous to those of \cite{DiMi90} in
the same situation. Our main result is the following:

\begin{theorem}\label{thm:main}
Let $\bG$ be a connected reductive group and $F\colon \bG\rightarrow \bG$ a
Frobenius endomorphism, and let  $G=\bG^F$ be the corresponding group of Lie
type. Let $(\bG^\ast, F^\ast)$ be dual to $(\bG, F)$ and $s\in
G^\ast:=(\bG^\ast)^{F^\ast}$ a semisimple element such that
$\left(\rC_{\bG^\ast}(s)/\rC^\circ_{\bG^\ast}(s)\right)^{F^*}=1$, or
equivalently $\rC_{\bG^*}(s)^{F^*} \leqslant \rC_{\bG^*}^{\circ}(s)$. There 
exists a unique family of Jordan decomposition maps
$J_{s}\colon\mathcal{E}(G, s)\rightarrow \mathcal{E}(\rC_{G^\ast}(s), 1)$
satisfying properties (1)-(7) of Theorem~\ref{UniqueJord} below. Further, the 
collection $\{J_{s}\mid \rC_{\bG^*}(s)^{F^*} \leqslant \rC_{\bG^*}^{\circ}(s)\}$ 
is $\mathcal{G}$-equivariant (in the sense of Lemma~\ref{galequiv} below).
\end{theorem}

In
the context of the theorem, we may embed the underlying connected reductive group $\bG$ into
another connected reductive group $\wt{\bG}$ with a connected center, using a
so-called regular embedding $\iota \colon \bG\hookrightarrow \wt{\bG}$, which in
turn yields a dual surjection $\iota^\ast\colon \wt{\bG}^\ast \rightarrow
\bG^\ast$ of the dual groups. Our collection of bijections $J_s$ is defined using such a regular embedding.
We show that this map is independent of the choice of preimage $\tilde{s}\in(\wt\bG^\ast)^{F^\ast}$ 
such that $\iota^\ast(\tilde{s})=s$ and, more remarkably, that this map is even independent 
of the choice of regular embedding (see Proposition \ref{prop:bumper-Js} below).

Along the way, we prove results about multiplicity-free restrictions (see Section \ref{sec:multfree}), 
extending impactful results of Lusztig, and results concerning the interaction 
of Jordan decomposition and Deligne-Lusztig induction with isotypies (see Sections \ref{sec:DLandisotypies} 
and \ref{sec:JD-and-isotypies} and Theorem \ref{thm:prop-6-DM}). We believe these results may be of independent interest. 

\subsection{Notation}
For any group $G$ and element $g \in G$ we denote by $\Ad_g : G \to G$ the inner
automorphism defined by $\Ad_g(x) = {}^gx = gxg^{-1}$. We also write $[g,x] =
gxg^{-1}x^{-1}$ for the commutator of $g,x \in G$. If $H \leqslant G$ is a
subgroup, then by restriction we obtain an isomorphism $H \to {}^gH = gHg^{-1}$,
which we also denote by $\Ad_g$.
 
Suppose now that $G$ is finite. We write $\cf(G)$ for the space of
complex-valued class functions on $G$ and $\irr(G) \subseteq \cf(G)$ for the set
of complex irreducible characters of $G$. We let $\triv\in\irr(G)$ denote the
trivial character of $G$. If $f \in \cf(G)$ is a class function then we write
$\irr(G \mid f)$ for the irreducible constituents of $f$. 
   
If $\phi : G\to H$ is a homomorphism between finite groups, we write
$\tw{\top}\phi$ for the function $\tw{\top}\phi\colon \cf(H)\rightarrow\cf(G)$
defined by $\tw{\top}\phi(\chi) = \chi \circ \phi$. If $\phi$ is injective and
we identify $G$ with $\phi(G)$, then this map can be viewed as restriction;
similarly if $\phi$ is surjective then this map can be viewed as inflation
through the quotient map $G \to G/{\ker\phi}$.

Finally if $G$ and $H$ are finite groups then for any $\chi \in \irr(G)$ and
$\psi \in \irr(H)$ we write $\chi\boxtimes\psi \in \irr(G\times H)$ for the
character defined by $(\chi\boxtimes\psi)(g,h) = \chi(g)\psi(h)$ for any $(g,h)
\in G\times H$.

\section{Generalizing Digne--Michel's Unique Jordan Decomposition}\label{sec:DMmap}
In this section, we develop some basic notation and state our main result.
Throughout, $p>0$ will be a fixed prime integer and $\FF = \overline{\FF}_p$
will be an algebraic closure of the finite field of cardinality $p$. All
algebraic groups are assumed to be affine $\FF$-varieties.

If $\bT$ is a torus, then we denote by $X(\bT)$ and $\wc{X}(\bT)$ the character
and cocharacter groups of $\bT$. Recall that for any two tori $\bT$ and $\bT'$
we have bijections
%%%%
\begin{equation*}
    \Hom(\bT,\bT')
    \overset{X(-)}{\longrightarrow}
    \Hom_{\ZZ}(X(\bT'),X(\bT))
    \overset{\wc{(-)}}{\longrightarrow}
    \Hom_{\ZZ}(\wc X(\bT),\wc X(\bT'))
    \overset{\wc{X}(-)}{\longleftarrow}
    \Hom(\bT,\bT')
\end{equation*}
%%%%
where $\Hom(\bT,\bT')$ denotes the set of homomorphisms of algebraic groups. If
$\phi \in \Hom(\bT,\bT')$ then $X(\phi)$ is the map $\chi \mapsto \chi\circ\phi$
and $\wc{X}(\phi)$ is the map $\gamma \mapsto \phi\circ\gamma$.

If $\bG$ is connected reductive and $\bT \leqslant \bG$ is a maximal torus, then
we denote by $\Phi_{\bG}(\bT) \subseteq X(\bT)$ and $\wc\Phi_{\bG}(\bT)
\subseteq \wc{X}(\bT)$ the roots and coroots of $\bG$, respectively. We also
call $W = \rN_{\bG}(\bT)/\bT$ the Weyl group of $\bG$ (with respect to $\bT$).

Now let $(\bG,F)$ be a pair consisting of an algebraic group $\bG$ and a
Frobenius endomorphism $F: \bG \to \bG$. The morphism $\Lang = \Lang_F =
\Lang_{\bG,F} : \bG \to \bG$, defined by $\Lang(g) = g^{-1}F(g)$, is called the
Lang map of $(\bG,F)$. It is surjective when $\bG$ is connected. If $\bG$ is
connected reductive, then we refer to $(\bG,F)$, or the finite group of fixed
points $G = \bG^F \leqslant \bG$, as a finite reductive group. We denote by
$\varepsilon_{(\bG,F)} = \varepsilon_{\bG} \in \{1,-1\}$ the sign defined in
\cite[Def.~7.1.5]{dmbook2}.

Following Steinberg \cite[p.~78]{St68} we say $(\bG,F)$ is $F$-simple if $\bG =
\bG_1\cdots\bG_n$ is an almost direct product of quasisimple groups permuted
cyclically by $F$. We refer to the type of an $F$-simple group $(\bG,F)$ as the
type of $(\bG_1,F^n)$, which is the type of the underlying root system of
$\bG_1$ decorated by the order of the automorphism induced by $F^n$. In general,
the type of $(\bG,F)$ is the product of the types of its $F$-simple components.

Let $(\bG,F)$ and $(\bG^{\ast},F^{\ast})$ be two finite reductive groups and
assume $(\bT,\bT^{\ast},\delta)$ is a triple consisting of: an $F$-stable
maximal torus $\bT \leqslant \bG$, an $F^{\ast}$-stable maximal torus
$\bT^{\ast} \leqslant \bG^{\ast}$, and an isomorphism $\delta : X(\bT) \to
\wc{X}(\bT^{\ast})$ satisfying
%%%%
\begin{equation}\label{eq:commute-frob}
    \wc{X}(F^{\ast}) \circ \delta = \delta\circ X(F).
\end{equation}
%%%%
We say $(\bG,F)$ and $(\bG^{\ast},F^{\ast})$ are \emph{dual} if for some triple
$\scT = (\bT,\bT^{\ast},\delta)$ we have $\delta$ is an isomorphism of root
data, see \cite[Def.~11.1.10]{dmbook2}. We call $\scT$ a \emph{witness} to the
duality and $\cD = ((\bG,F),(\bG^{\ast},F^{\ast}),\scT)$ a \emph{rational
duality}. Note that this induces a duality between the corresponding Weyl groups
$W = \rN_{\bG}(\bT)/\bT$ and $W^* = \rN_{\bG^*}(\bT^*)/\bT^*$. 

Let us fix once and for all an embedding $\FF^{\times} \hookrightarrow
\CC^{\times}$ and an isomorphism $\FF^{\times} \to (\QQ/\ZZ)_{p'}$. If two tori
$(\bT,F)$ and $(\bT^{\ast},F^{\ast})$ are dual, then each isomorphism $\delta :
X(\bT) \to \wc{X}(\bT^{\ast})$ satisfying \eqref{eq:commute-frob} determines a
group isomorphism $\bT^{* F^*} \to \irr(\bT^F)$ which we denote by $s \mapsto
\hat s$, see \cite[Prop.~11.1.14]{dmbook2}. This depends on $\delta$ and our
preceding choices of embedding $\FF^{\times} \hookrightarrow \CC^{\times}$ and
isomorphism $\FF^{\times} \to (\QQ/\ZZ)_{p'}$.

Assume we have a rational duality and let $G^\ast := \bG^{\ast F^\ast}$ be the
finite dual group. Given a semisimple element $s\in G^\ast$, we denote by
$\cE(G, s)\subseteq \irr(G)$ the \emph{rational Lusztig series} corresponding to
the $G^\ast$-conjugacy class of $s$, see \cite[Def.~12.4.3]{dmbook2}. If $z \in
\rZ(G^{\ast}) \leqslant \bT^{\ast F^{\ast}}$ is a central element, then
$\cE(G,z)$ contains a linear character $\hat z$ whose restriction to $\bT^F$ is
the character also denoted by $\hat z$ above, see \cite[Prop.~11.4.12]{dmbook2}.
We have $\cE(G,z) = \cE(G,1)\otimes \hat{z}$.

Lusztig \cite{Lu88} has shown that there exists a Jordan decomposition map
$\cE(G, s)  \rightarrow \cE(\rC_{G^\ast}(s), 1)$ (that is, a map satisfying
condition (1) of Theorem \ref{UniqueJord} below) but such a map is not unique in
general. Generalising a result of Digne and Michel \cite[Thm.~7.1]{DiMi90} we
will show that, when $\rC_{G^{\ast}}(s) \leqslant \rC_{\bG^{\star}}^{\circ}(s)$,
this map (or rather, this family of maps, per the conditions below) can be
chosen uniquely to satisfy the following properties.

\begin{theorem}\label{UniqueJord}
Let $\cD = ((\bG,F),(\bG^{\ast},F^{\ast}),\scT_0)$ be a rational duality, with
$\scT_0 = (\bT_0,\bT_0^{\ast},\delta_0)$, and let $\mathcal{S}_\cD\subseteq G^\ast$ be the set of semisimple elements $s \in
G^{\ast}$ satisfying $\rC_{G^{\ast}}(s) \leqslant \rC_{\bG^{\star}}^{\circ}(s)$. Then there exists
 a unique collection of bijections
%%%%
\begin{equation*}
   J_s^{\bG} = J_s^{\cD} : \cE(G, s) \longrightarrow \cE(\rC_{G^*}(s), 1) 
\end{equation*}
%%%%
indexed by $s\in \mathcal{S}_\cD$
such that the following properties hold:
%%%%
\begin{enumerate}
    \item If $\bT \leqslant \bG$ and $\bT^{\ast} \leqslant \bG^{\ast}$ are dual
    maximal tori that are $F$-stable and $F^{\ast}$-stable, respectively, and
    $\bT^{\ast} \leqslant \rC_{\bG^{\ast}}^{\circ}(s)$, then for any $\chi \in
    \cE(G, s)$ we have
    %%%%
    \begin{equation*}
    \langle \chi, R_{\bT}^{\bG}(\wh{s}) \rangle
    =
    \varepsilon_{\bG} \varepsilon_{\rC_{\bG^*}^{\circ}(s)}
    \langle J_s^{\bG}(\chi), R_{\bT^*}^{\rC_{\bG^*}^{\circ}(s)} (\triv) \rangle.
    \end{equation*}
    %%%%

    \item If $s=1$ then:
    %%%%
    \begin{enumerate}
        \item Let $d$ be the smallest positive integer such that $F^d$ is a
        split Frobenius endomorphism, see \cite[Def.~4.3.2]{dmbook2}.  The
        eigenvalues of $F^d$ associated to $\chi$ are equal, up to an integer
        power of $q^{d/2}$, to the eigenvalues of $F^{*d}$ associated to
        $J_1^{\bG}(\chi)$.
        
        \item If $\chi$ is in the principal series, then $J_1^{\bG}(\chi)$ and
        $\chi$ correspond to the same character of the Hecke algebra.
    \end{enumerate}

    \item If $z \in \rZ(G^*)$ is central and $\chi \in \cE(G, s)$, then
    %%%%
    \begin{equation*}
    J_{sz}^{\bG}(\chi \otimes \hat{z}) = J_s^{\bG}(\chi).
    \end{equation*}

    \item  If $\bL \leqslant \bG$ and $\bL^{\ast} \leqslant \bG^*$ are dual Levi
    subgroups that are $F$-stable and $F^{\ast}$-stable, respectively, and
    $\rC_{\bG^*}^{\circ}(s) \leqslant \bL^{\ast}$, then extending linearly we have
    %%%%
    \begin{equation*}
    J_s^{\bL} = J_s^{\bG} \circ R_{\bL}^{\bG}
    \end{equation*}
    %%%%
    as maps $\ZZ\cE(L,s) \to \ZZ\cE(\rC_{G^{\star}}(s),1)$, where
    $R_{\bL}^{\bG}$ denotes Lusztig's twisted induction.

    \item Assume $(\bG, F)$ is $F$-simple of type $E_8$ and $(\rC_{\bG^*}(s),
    F^*)$ is of type $E_7.A_1$ (respectively, $E_6.A_2$, respectively ${^2 E_6}.
    {^2 A_2}$). If $\bL \leqslant \bG$ and $\bL^{\ast} \leqslant \bG^*$ are dual
    Levi subgroups of type $E_7$ (respectively, $E_6$, respectively $E_6$) that
    are $F$-stable and $F^{\ast}$-stable, respectively, and $s \in
    \rZ(\bL^{\ast})$, then the following diagram is commutative:
    %%%%
    \begin{center}
    \begin{tikzcd}[sep=1cm]
    \ZZ\cE(G, s)
    \arrow[r,"J_s^{\bG}"]
    & 
    \ZZ\cE(\rC_{G^*}(s), 1)
    \\
    \ZZ\cE(L, s)^{\bullet}
    \arrow[r,"J_s^{\bL}"]
    \arrow[u,"R_{\bL}^{\bG}"]
    &
    \ZZ\cE(L^{\ast}, 1)^{\bullet}
    \arrow[u,swap,"R_{\bL^{\ast}}^{\rC_{\bG^{\ast}}(s)}"]
    \end{tikzcd}
    \end{center}
    %%%%
    where the superscript $\bullet$ denotes the cuspidal part of the Lusztig
    series. 

    \item For any isotypy $\varphi : (\bG,F) \to (\bG_1,F_1)$, with dual
    $\varphi^{\ast} : (\bG_1^{\ast},F_1^{\ast}) \to (\bG^{\ast},F^{\ast})$ and
    any semisimple element $s_1 \in G_1^{\star}$ satisfying $s =
    \varphi^*(s_1)$, the following diagram commutes
    %%%%
    \begin{center}
    \begin{tikzcd}[sep=1cm]
    \cE(G, s)
    \arrow[r,"J_s^{\bG}"]
    & 
    \cE(\rC_{G^*}(s), 1)
    \arrow[d,"\tw{\top}\varphi^{\ast}"]
    \\
    \cE(G_1, s_1)
    \arrow[r,"J_{s_1}^{\bG_1}"]
    \arrow[u,"\tw{\top}\varphi"]
    &
    \cE(\rC_{G_1^*}(s_1), 1)
    \end{tikzcd}
    \end{center}
    %%%%

    \item If $\bG = \prod_i \bG_i$ is a direct product of $F$-stable subgroups,
    then $J_{\prod_i s_i}^{\bG} = \prod_i J^{\bG_i}_{s_i}$.
\end{enumerate}
\end{theorem}
In condition (2b) above, by the Hecke algebra, we mean $\mathrm{End}_{\mathbb{C}
G}(\ind_B^G(\triv))$,  where $B=\mathbf{B}^F$ and $\mathbf{B}$ is an $F$-stable
Borel subgroup of $\mathbf{G}$, where the characters of this Hecke algebra are
in natural bijection with the characters of the Hecke algebra for the
corresponding dual groups through an isomorphism of Lusztig \cite{Lu81}. The
definition of an isotypy, used in (6), is recalled in
Section~\ref{sec:multfree}.

The properties in Theorem~\ref{UniqueJord} parallel those of
\cite[Thm.~7.1]{DiMi90}. Indeed, our bijection is built from the bijection
constructed in \cite{DiMi90}. However, both condition (2a) and (6) of
Theorem~\ref{UniqueJord} are stronger than the corresponding conditions in
\cite[Thm.~7.1]{DiMi90}. In \cite{SrVi15} Srinivasan--Vinroot have shown that
Digne--Michel's bijection satisfies (2a) of Theorem~\ref{UniqueJord} and we will
show in Theorem~\ref{thm:prop-6-DM} that it also satisfies (6). As an
epimorphism with kernel a central torus is certainly an isotypy, we have the
following.

\begin{lemma}\label{lem:Js-is-DM-bij}
If $\rZ(\bG)$ is connected, then any collection of bijections satisfying the properties listed in Theorem~\ref{UniqueJord} must be
Digne--Michel's unique Jordan decomposition defined in \cite{DiMi90}.
\end{lemma}

We close by making a few remarks. Let $\scT_0 = (\bT_0,\bT_0^{\star},\delta_0)$
be the witness to the duality occurring in the statement of
Theorem~\ref{UniqueJord}. In (1) of the theorem, we must choose an isomorphism
$\delta : X(\bT) \to \wc{X}(\bT^{\star})$ for the map $s \mapsto \hat s$ to be
defined. In other words, we must choose a witness $(\bT,\bT^{\ast},\delta)$ to
the duality of $(\bT,F)$ and $(\bT^*,F^*)$. In (4) and (5) of the statement, we
must choose a witness to the duality of $(\bL,F)$ and $(\bL^*,F^*)$ for the
bijection $J_s^{\bL}$ to be defined. We briefly recall how this is done, by
inheritence from $\scT_0$.

For any $(g,g^*) \in \bG \times \bG^*$ we may consider the tuple
%%%%
\begin{equation*}
(g,g^*)\cdot\scT_0 
= 
({}^g\bT_0,{}^{g^*}\bT_0^{\star},\wc{X}(\Ad_{g^*})\circ\delta_0\circ X(\Ad_g)).
\end{equation*}
%%%%
In general, this will not be a witness to the duality between $(\bG,F)$ and
$(\bG^*,F^*)$ because \eqref{eq:commute-frob} may not be satisfied. It is a
witness exactly when \eqref{eq:commute-frob} holds and this condition may be
rephrased in terms of the Weyl group as in \cite[Lem.~4.3.3]{Ca85}. In (1), (4),
and (5), we always assume the witness is of the form $(g,g^*)\cdot \scT_0$.

More precisely, saying that two Levi subgroups $\bL \leqslant \bG$ and $\bL^*
\leqslant \bG^*$ are dual means exactly that some $\scT = (g,g^*)\cdot\scT_0$ is
a witness to the duality between $(\bL,F)$ and $(\bL^*,F^*)$. If $\scT =
(\bT,\bT^*,\delta)$, then this is equivalent to requiring that: $\scT$ is a
witness to the duality between $(\bG,F)$ and $(\bG^*,F^*)$, $\bT \leqslant \bL$,
$\bT^* \leqslant \bL^*$, and $\delta(\Phi_{\bL}(\bT)) = \wc\Phi_{\bL^*}(\bT^*)$.

Finally we must explain how to choose witnesses in (6) and (7). In (6) we assume
that $\scT_1 = (\bT_1,\bT_1^{\ast},\delta_1)$ is a witness to the duality
between $(\bG_1,F_1)$ and $(\bG_1^{\ast},F_1^{\ast})$ satisfying $\varphi(\bT_0)
\leqslant \bT_1$. We then have $\varphi^{\ast}$ is dual to $\varphi$ in the
sense of Definition~\ref{def:dual-isotypy}. In (7) we assume $\bT = \prod_i
\bT_i$ and $\bT^{\ast} = \prod_i \bT_i^{\ast}$ are products of $F$-stable
maximal tori and $\delta = \prod_i \delta_i$ where $\delta_i : X(\bT_i) \to
\wc{X}(\bT_i^{\ast})$ is an isomorphism making $(\bT_i,\bT_i^{\ast},\delta_i)$ a
witness to the duality between $(\bG_i,F)$ and $(\bG_i^{\ast},F)$.

\section{Multiplicity Free Restrictions}\label{sec:multfree}
In this section, we use a difficult result of Lusztig on spin groups to show
that restriction from $G$ to $\rO^{p'}(G)$ is multiplicity free. From this, we
conclude Lusztig's multiplicity freeness result \cite[Prop.~10]{Lu88}. Namely,
this says that if $\bG \to \wt{\bG}$ is a regular embedding, as defined below,
then restriction from $\wt{G} = \wt{\bG}^F$ to $G$ is multiplicity free.

We note that Li \cite[Lem.~2.1]{Li23} has shown that the restriction map
$\res_{[G,G]}^G$ is multiplicity free assuming ``$q$ is large enough''. From the
proof in \cite{Li23} we see this is meant to mean that $\rO^{p'}(G) = [G,G]$.
The proof in \cite[Lem.~2.1]{Li23} relies on \cite[Prop.~10]{Lu88}, whereas our
proof does not utilise regular embeddings at all. We believe the approach taken
here may be of interest for other reductions in the future.

As in Lusztig's original approach \cite{Lu88}, we need to reduce to the case
where $\bG$ is simple and simply connected. It is not enough to prove the
statement in this case (which is trivial because $G = \rO^{p'}(G)$ when $\bG$ is
simply connected), so we need to prove a different statement. For this, we wish
to consider finite overgroups of $G$ contained in the normalizer $\rN_{\bG}(G)$.

It is well known that the centralizer $\rC_{\bG}(G) = \rZ(\bG)$ is the center of
$\bG$, see \cite[Lem.~6.1]{Bon00}. The normalizer may also be described in terms
of the center.

\begin{lemma}\label{lem:normalizer-is-lang-preimage}
We have $\rN_{\bG}(G) = \Lang^{-1}(\rZ(\bG))$.
\end{lemma}

\begin{proof}
If $g \in \rN_{\bG}(G)$, then for any $x \in G$ we have ${}^gx \in G$. In
partiular, ${}^{F(g)}x = {}^gx$ for any $x \in G$ so $\Lang(g) \in \rC_{\bG}(G)
= \rZ(\bG)$, hence $\rN_{\bG}(G) \leqslant \Lang^{-1}(\rZ(\bG))$. Since
$\rZ(\bG)$ is $F$-stable, $\Lang^{-1}(\rZ(\bG)) \leqslant \rN_{\bG}(G)$.
\end{proof}

Let us draw some conclusions from this equality. Firstly, the natural map
$\rN_{\bG}(G)/\rC_{\bG}(G) \to (\bG/\rZ(\bG))^F$ is an isomorphism, hence the
automizer $\rN_{\bG}(G)/\rC_{\bG}(G)$ is a finite group. The image of the
natural map $\rN_{\bG}(G) \to \Aut(G)$ is the group of inner diagonal
automorphisms, as defined in \cite[\S11.5]{dmbook2}.

Recall that $\bG = \bG_{\der}\cdot \rZ^{\circ}(\bG)$ , where $\bG_{\der}
\leqslant \bG$ is the derived subgroup of $\bG$ \cite[\S1.8]{Ca85}. If $\rN_{\bG_{\der}}(G) :=
\bG_{\der} \cap \rN_{\bG}(G)$ then we have a natural map $\rN_{\bG_{\der}}(G)
\to \rN_{\bG}(G)/\rC_{\bG}(G)$ whose kernel $\rC_{\bG_{\der}}(G) := \bG_{\der}
\cap \rC_{\bG}(G) = \rZ(\bG_{\der})$ is finite.

Finally, by the Lang--Steinberg Theorem, the Lang map defines an isomorphism of
abstract groups $\rN_{\bG}(G)/G \to \rZ(\bG)$. The image of the subgroup
$(G\cdot \rZ(\bG))/G$ is the image $\Lang(\rZ(\bG))$ of the Lang map. Hence, we
have an isomorphism
%%%%
\begin{equation*}
\rN_{\bG}(G)/(G\cdot \rZ(\bG)) \cong \rZ(\bG)/\Lang(\rZ(\bG)).
\end{equation*}
%%%%
As $G$ is finite, there is a bijection $A \mapsto \Lang^{-1}(A)$ between the
finite subgroups of $\rZ(\bG)$ and the finite overgroups $X \leqslant
\rN_{\bG}(G)$ of $G$.

\begin{lemma}\label{lem:overgroups-of-G}
If $G_{\der} = (\bG_{\der})^F$ then the following hold:
%%%%
\begin{enumerate}[label={\normalfont(\roman*)},topsep=5pt,itemsep=5pt]
    \item $\rN_{\bG_{\der}}(G) = \rN_{\bG_{\der}}(G_{\der})$ is finite and
    $\rN_{\bG}(G) = \rN_{\bG_{\der}}(G) \cdot \rC_{\bG}(G)$,
    \item $\Lang^{-1}(\rZ(\bG_{\der})) = G\cdot \rN_{\bG_{\der}}(G)$ is a finite
    overgroup of $G$,
    \item if $X \leqslant \rN_{\bG}(G)$ is a finite overgroup of $G$, then $X
    \leqslant \rN_{\bG_{\der}}(G) \cdot Z$ for some finite subgroup $Z \leqslant
    \rC_{\bG}(G)$.
\end{enumerate}
\end{lemma}

\begin{proof}
(i). As $\rZ(\bG_{\der}) \leqslant \rZ(\bG)$ we have by
Lemma~\ref{lem:normalizer-is-lang-preimage} that 
%%%%
\begin{equation*}
\rN_{\bG_{\der}}(G) = \bG_{\der} \cap \Lang^{-1}(\rZ(\bG)) = \bG_{\der} \cap
\Lang^{-1}(\rZ(\bG_{\der})) = \rN_{\bG_{\der}}(G_{\der})
\end{equation*}
%%%%
which is finite since $\rZ(\bG_{\der})$ is. That $\rN_{\bG}(G) = \rN_{\bG_{\der}}(G) \cdot
\rC_{\bG}(G)$ follows from the fact that $\bG = \bG_{\der} \cdot
\rZ^{\circ}(\bG)$.

(ii). By the Lang--Steinberg Theorem $\Lang(\rN_{\bG_{\der}}(G)) =
\rZ(\bG_{\der})$ and the fiber of the Lang map over any point is a coset of the
finite group $G$ in $\bG$. This gives the equality.

(iii). We may take $Z$ to be the inverse image of the finite group
$(X\cdot \rN_{\bG_{\der}}(G))/\rN_{\bG_{\der}}(G)$ under the natural surjective
map $\rC_{\bG}(G) \twoheadrightarrow \rN_{\bG}(G)/\rN_{\bG_{\der}}(G)$.
\end{proof}

Recall that a morphism of algebraic groups $\phi : \bG \to \bG'$ is an
\emph{isotypy} if $\bG_{\der}' \leqslant \phi(\bG)$ and $\ker(\phi) \leqslant
\rZ(\bG)$. An isotypy $\phi : (\bG,F) \to (\bG',F')$ is an isotypy $\phi : \bG
\to \bG'$ satisfying $\phi\circ F = F'\circ\phi$. When using isotypies, we will
always assume that both $\bG$ and $\bG'$ are connected reductive algebraic
groups. This has the consequence that $\bG' = \phi(\bG)\cdot\rZ^{\circ}(\bG')$
for any isotypy $\phi : \bG \to \bG'$.

It will be useful at several points to have some properties regarding the
normalizer $\rN_{\bG}(G)$ with respect to isotypies.

\begin{lemma}\label{lem:isotypy-der-diag}
If $\phi : (\bG,F) \to (\bG',F')$ is an isotypy and $G = \bG^F$ and $G' =
\bG'^{F'}$ then the following hold:
%%%%
\begin{enumerate}[label={\normalfont(\roman*)},topsep=5pt,itemsep=5pt]
    \item $\phi(\bG_{\der}) = \bG_{\der}'$,
    \item $\phi^{-1}(\rZ(\bG')) = \rZ(\bG)$,
    \item $\phi(\rN_{\bG}(G)) = \rN_{\phi(\bG)}(G')$,
    \item $\phi(\Lang_{\bG,F}^{-1}(\rZ(\bG_{\der}))) =
    \Lang_{\phi(\bG),F'}^{-1}(\rZ(\bG_{\der}'))$.
\end{enumerate}
\end{lemma}

\begin{proof}
(i). As $\bG' = \phi(\bG)\cdot \rZ^{\circ}(\bG')$ we have $\bG'_{\der} =
\phi(\bG)_{\der} = \phi(\bG_{\der})$.

(ii). Let $A = \ker(\phi) \cap \bG_{\der} \leqslant \rZ(\bG_{\der}) \leqslant
\rZ(\bG)$. If $g \in \phi^{-1}(\rZ(\bG'))$ then the commutator $[g,-] : \bG \to
A$ defines a morphism of algebraic groups. Now $[g,\bG] \leqslant A$ is
connected, because $\bG$ is, and $A$ is finite so $g \in \rZ(\bG)$. This shows
that $\phi^{-1}(\rZ(\bG')) \leqslant \rZ(\bG)$ and equality holds because $\bG'
= \phi(\bG)\cdot \rZ^{\circ}(\bG')$.

(iii). Assume $g \in \bG$ is such that $\phi(g) \in \rN_{\bG'}(G')$. Then
$\phi(\Lang_{\bG,F}(g)) = \Lang_{\bG',F'}(\phi(g)) \in \rZ(\bG')$ and so
$\Lang_{\bG,F}(g) \in \rZ(\bG)$ by (ii), which shows that $g \in \rN_{\bG}(G)$.

(iv). Using (i) we may argue exactly as in (iii).
\end{proof}

We next consider extendability in the case of semisimple groups.

\begin{proposition}\label{thm:ext-stab-Op'-der}
If $\bG = \bG_{\der}$ is semisimple, then any character $\chi \in
\irr(\rO^{p'}(G))$ extends to its stabilizer $\rN_{\bG}(G)_{\chi}$ in the finite
group $\rN_{\bG}(G)$.
\end{proposition}

\begin{proof}
Let $\phi : \bG_{\simc} \twoheadrightarrow \bG$ be a simply connected covering
map and write $G_{\simc}:=\bG_{\simc}^F$. Here we have $F$ extends uniquely to a
Frobenius endomorphism $F : \bG_{\simc} \to \bG_{\simc}$, see \cite[9.16]{St68}
and the remarks in \cite[9.17]{St68}. As $\phi$ is a bijection on unipotent
elements $\phi(\rO^{p'}(G_{\simc})) = \rO^{p'}(G)$. By
Lemma~\ref{lem:isotypy-der-diag} $\phi(\rN_{\bG_{\simc}}(G_{\simc})) =
\rN_{\bG}(G)$ so if $\psi := \tw{\top}\phi(\chi) \in \irr(\rO^{p'}(G_{\simc}))$
extends to its stabilizer $\rN_{\bG_{\simc}}(G_{\simc})_{\psi}$ then deflating
this extension gives an extension of $\chi$ to $\rN_{\bG}(G)_{\chi}$.

Therefore, we can assume that $\bG = \bG_{\simc}$ is simply connected, which
means $G = \rO^{p'}(G)$ by a theorem of Steinberg \cite[Thm.~12.4]{St68}. Suppose
$\bG = \bG^{(1)} \times \cdots \times \bG^{(r)}$ is an $F$-stable decomposition,
where $\bG^{(i)}$ is a product of quasisimple groups. As $\rN_{\bG}(G) =
\rN_{\bG^{(1)}}(G^{(1)}) \times \cdots \times \rN_{\bG^{(r)}}(G^{(r)})$, we may
assume $\bG = \bG_1 \times \cdots \times \bG_n$ is a product of quasisimple
groups permuted transitively by $F$.

Let $\pi : \bG \twoheadrightarrow \bG_1$ be the natural projection map. If $g
\in \rN_{\bG}(G)$, then
%%%%
\begin{equation*}
    g\rZ(\bG) \in (\bG/\rZ(\bG))^F 
    \leqslant 
    (\bG/\rZ(\bG))^{F^n} 
    \cong
    (\bG_1/\rZ(\bG_1))^{F^n} \times \cdots \times (\bG_n/\rZ(\bG_n))^{F^n}
\end{equation*}
%%%%
so $\pi(g) \in \rN_{\bG_1}(G_1)$ where $G_1 = \bG_1^{F^n}$. Moreover, if $h \in
\rN_{\bG_1}(G_1)$ then $hF(h)\cdots F^{n-1}(h) \in \rN_{\bG}(G)$ so $\pi$
restricts to a surjective homomorphism $\rN_{\bG}(G) \twoheadrightarrow
\rN_{\bG_1}(G_1)$ which further restricts to an isomorphism $G
\overset{\sim}{\to} G_1$. Identifying $\chi$ with a character of $G_1$, we see
that if $\chi$ extends to $\rN_{\bG_1}(G_1)_{\chi}$ then inflating we get an
extension of $\chi$ to $\rN_{\bG}(G)$. Hence, we can assume that $\bG$ is
quasisimple and simply connected.

Now if the quotient $\rN_{\bG}(G)/G \cong \rZ(\bG)/\Lang(\rZ(\bG))$ is cyclic,
then $\chi$ will extend to its stabilizer. This is the case unless $F$ is split,
$\bG = \mathrm{Spin}_{4n}(\FF)$ is a spin group, and $q$ is odd. But this very
tricky case has been dealt with by a counting argument due to Lusztig. A
detailed proof of this statement appears in \cite[Thm.~5.11]{CaEn04} and
\cite{Lu08}.
\end{proof}

With this, we can now establish the desired extendibility statement for any
finite reductive group.

\begin{theorem}\label{thm:ext-stab-Op'}
If $X \leqslant \rN_{\bG}(G)$ is a finite overgroup of $\rO^{p'}(G)$, then any
character $\chi \in \irr(\rO^{p'}(G))$ extends to its stabiliser $X_{\chi}$.
\end{theorem}

\begin{proof}
Let $H = \rN_{\bG_{\der}}(G)$. Then by Lemma~\ref{lem:overgroups-of-G}, we have
$X \leqslant \hat{G} := HZ$ for some finite subgroup $Z \leqslant \rC_{\bG}(G) =
\rZ(\bG)$. It  suffices to show that $\chi$ extends to its stabilizer
$\hat{G}_{\chi}$. As $Z$ centralizes $H$, we have the product map $\pi :
H_{\chi} \times Z \twoheadrightarrow \hat{G}_{\chi}$ is a surjective group
homomorphism. By Proposition~\ref{thm:ext-stab-Op'-der}, $\chi$ has an extension
$\hat{\chi} \in \irr(H_{\chi})$ because $H = \rN_{\bG_{\der}}(G_{\der})$ by
Lemma~\ref{lem:overgroups-of-G}.

Now $H_{\chi} \cap Z \leqslant \rZ(\hat{G})$ so $\res_{H_{\chi} \cap
Z}^{H_{\chi}}(\hat{\chi}) = \hat{\chi}(1)\lambda$ for a unique $\lambda \in
\irr(H_{\chi} \cap Z)$. If $\eta \in \irr(Z)$ is an extension of $\lambda^{-1}$,
which exists because $Z$ is abelian, then we obtain an irreducible character
$\psi = \hat{\chi} \boxtimes \eta \in \irr(H_{\chi}\times Z)$ with $\ker(\pi)
\leqslant \ker(\psi)$. Deflating $\psi$ gives an extension of $\chi$ to
$\hat{G}_{\chi}$.
\end{proof}

\begin{theorem}\label{thm:multfree}
If $X \leqslant Y \leqslant \rN_{\bG}(G)$ are finite overgroups of
$\rO^{p'}(G)$, then restriction from $Y$ to $X$ is multiplicity free and every
character $\chi \in \irr(X)$ extends to its stabiliser $Y_{\chi}$.
\end{theorem}

\begin{proof}
Let $N = \rO^{p'}(G)$. Any character $\psi \in \irr(N)$ extends to its
stabilizer $Y_{\psi}$ by Theorem~\ref{thm:ext-stab-Op'}. As $Y/N$ is abelian, we
have by Gallagher's Theorem that $\ind_N^{Y_{\psi}}(\psi)$ is multiplicity free,
hence so is $\ind_N^Y(\psi)$ by Clifford's correspondence. Frobenius reciprocity
now implies that restriction from $Y$ to $N$ is multiplicity free and thus
restriction from $Y$ to $X$ must also be multiplicity free. For the last
statement we reverse the argument using Frobenius reciprocity and Clifford's
correspondence to conclude that $\ind_X^Y(\chi)$ and hence
$\ind_X^{Y_{\chi}}(\chi)$ are multiplicity free. Frobenius reciprocity now shows
that $\chi$ extends to $Y_{\chi}$.
\end{proof}

As in \cite{Lu88}, we define a \emph{regular embedding} to be an injective
homomorphism $\iota \colon \bG\hookrightarrow\wt{\bG}$, where $(\wt{\bG},
\wt{F})$ is a finite reductive group such that $\rZ(\wt{\bG})$ is connected,
$\iota$ commutes with the Frobenius morphisms in the sense that $\iota\circ
F=\wt{F}\circ \iota$, the map $\iota$ induces an isomorphism of $\bG$ onto the
closed subgroup $\iota(\bG)$ of $\wt{\bG}$, and $\iota(\bG)_{\der} =
\wt\bG_{\der}$. Given such a map, there is a corresponding dual surjection $
\iota^\ast\colon \wt{\bG}^\ast \rightarrow \bG^\ast$, such that
$\ker(\iota^\ast)$ is a central torus, see Definition~\ref{def:dual-isotypy} for
further details.

If $\bG \to \wt\bG$ is a regular embedding then, identifying $\bG$ with its
image and writing $\wt G:=\wt\bG^{\wt{F}}$, we have $\rO^{p'}(\wt G) \leqslant G
\leqslant \wt{G} \leqslant \rN_{\wt\bG}(\wt G)$ so from this we obtain Lusztig's
result \cite[Prop.~10]{Lu88}. The following shows that the usual information one
obtains from a regular embedding can be read off from the finite overgroup
$G\cdot\rN_{\bG_{\der}}(G) \leqslant \rN_{\bG}(G)$ of $G$.

\begin{lemma}\label{prop:stabhatG}
Assume $\iota : (\bG,F) \to (\wt{\bG},F)$ is an injective isotypy and
$\rZ(\wt{\bG})$ is connected. Let $\hat{G} = G\cdot \rN_{\bG_{\der}}(G)$. For
any $\chi \in \irr(G)$, we have an isomorphism $\wt{G}/\wt{G}_{\chi} \cong
\hat{G}/\hat{G}_{\chi}$ where $\wt{G} = \wt{\bG}^F$.
\end{lemma}

\begin{proof}
We identify $\bG$, as an abstract group, with its image in $\wt\bG$. Let $\Gamma
= \hat{G}Z$ where $Z = \{z \in \rZ(\wt{\bG}) \mid \Lang(z) \in
\rZ(\bG_{\der})\}$. It follows from the Lang--Steinberg Theorem and the
decomposition $\wt\bG = \bG_{\der}\cdot \rZ(\wt\bG)$ that $\Gamma = \wt{G}Z$.
This implies $\Gamma_{\chi} = \wt{G}_{\chi}Z = \hat{G}_{\chi}Z$, which yields
the statement.
\end{proof}

\section{Isotypies and Deligne--Lusztig Induction}\label{sec:DLandisotypies}

In this section, we develop further results on isotypies. For this, we will need
the following result on Deligne--Lusztig induction, which generalizes a standard
result found in \cite[Prop.~11.3.10]{dmbook2}. Related statements on 2-variable
Green functions and bounded derived categories are found in
\cite[Prop.~2.2.2]{Bon00} and \cite[Prop.~1.1]{BR06}.

Before stating the result, let us introduce the following notation. If $\bX$ is
a variety and $g \in \Aut(\bX)$ is an element of finite order, then we denote by
%%%%
\begin{equation*}
\mathcal{L}(g \mid \bX) 
= \sum_{i \in \ZZ} 
    (-1)^i\Tr(g \mid H_c^i(\bX,\overline{\QQ}_{\ell}))
\end{equation*}
%%%%
the Lefschetz trace of $g$ acting on the cohomology of $\bX$.

\begin{proposition}\label{prop:surj-isotyp-DL}
Suppose $\phi : (\bG,F) \to (\bG',F')$ is an isotypy with kernel $\bK =
\ker(\phi) \leqslant \rZ(\bG)$ and let $\bL' \leqslant \bG'$ be an $F'$-stable
Levi complement of a parabolic subgroup $\bP' \leqslant \bG'$. If $(\bL,\bP) =
(\phi^{-1}(\bL'),\phi^{-1}(\bP'))$ then
%%%%
\begin{equation*}
{}^{\top}\phi \circ R_{\bL' \subseteq \bP'}^{\bG'} 
=
\frac{1}{\abs{\bK/\Lang(\bK)}}\sum_{z\Lang(\bK) \in
\bK/\Lang(\bK)}R_{\bL\subseteq\bP}^{\bG} \circ {}^{\top}\Ad_{l_z} \circ
{}^{\top}\phi
\end{equation*}
%%%%
where $l_z \in \bL$ is an element satisfying $\Lang(l_z) = z$.
\end{proposition}

\begin{proof}
Let $\bU \leqslant \bP$ be the unipotent radical of $\bP$, so  $\bU' =
\phi(\bU)$ is the unipotent radical of $\bP'$. We define $\bY' = \{g \in \bG'
\mid \Lang(g) \in \bU'\}$. We have
\cite[Lem. 9.1.5, Prop. 9.1.6]{dmbook2} that
%%%%
\begin{equation*}
R_{\bL'\subseteq\bP'}^{\bG'}(\chi)(\phi(g)) 
= 
\frac{1}{\abs{L'}}\sum_{l' \in L'}\mathcal{L}((\phi(g),l')\mid \bY')\chi(l'^{-1}).
\end{equation*}
%%%%
If $z \in \rZ(\bG)$ then we let $\bY_z = \{g \in \bG \mid \Lang(g) \in \bU z\}$,
which is a closed subset of $\bG$. If $l \in \Lang_{\bL,F}^{-1}(\rZ(\bG))$ then
$\bY_zl = \bY_{z\Lang(l)}$ for any $z \in \rZ(\bG)$.

We choose a finite subgroup $A \leqslant \bK$ such that $\bK = A \cdot
\Lang(\bK)$. This exists because $\bK$ is abelian and every element has finite
order. As $A$ is finite $\bY = \bigsqcup_{a \in A} \bY_a$ is a closed subset of
$\bG$ and $\phi$ factors through a bijective morphism $\bY/\hat{K} \to \bY'$,
where $\hat{K} = \bK \cap \Lang_{\bG,F}^{-1}(A)$ is a finite group.

The group $\hat{L} = \Lang_{\bL,F}^{-1}(A)$ is finite and satisfies
$\phi(\hat{L}) = L' := \bL'^F$. As $\bY_z \cap \bY_{z'} = \varnothing$ if $z \neq
z'$ we see that $\bY_al = \bY_a$ if and only if $l \in L$. We then have \cite[Prop. 8.1.10(ii), Prop. 8.1.7(iii)]{dmbook2} 
%%%%
\begin{equation*}
\mathcal{L}((g,l) \mid \bY/\hat{K}) 
= 
\frac{1}{\abs{\hat{K}}}\sum_{\substack{z \in \hat{K}\\ lz \in L}}\sum_{a \in A}
\mathcal{L}((g,l_a^{-1}lzl_a) \mid \bY_1)
\end{equation*}
%%%%
for any $(g,l) \in G \times \hat{L}^{\mathrm{opp}}$.

If $\pi : \hat{L} \times \hat{K} \to \bL$ is the natural product map then the
fiber $\pi^{-1}(g)$ over $g \in L$ has cardinality $\abs{\hat{K}}$. Summing over
$\hat{L}$ we get
%%%%
\begin{align*}
\frac{1}{\abs{\hat{L}}}\sum_{l \in \hat{L}} \mathcal{L}((g,l) \mid \bY/\hat{K}) 
= 
\frac{1}{\abs{A}}\sum_{a \in A} \frac{1}{\abs{L}}\sum_{l \in
L}\mathcal{L}((g,l_a^{-1}ll_a) \mid \bY_1),
\end{align*}
%%%%
and from this we conclude that
%%%%
\begin{equation*}
R_{\bL'\subseteq\bP'}^{\bG'}(\chi)(\phi(g)) 
= \frac{1}{\abs{A}}\sum_{a \in A}R_{\bL}^{\bG}(\chi \circ \phi \circ
\Ad_{l_a})(g).
\end{equation*}
%%%%
If $a \in A$ is contained in the kernel of the map $A \to \bK/\Lang(K)$ then
$\Ad_{l_a}$ restricts to an inner automorphism of $L$. Hence, we may take the
sum over $A/(A \cap \Lang(\bK)) \cong \bK/\Lang(\bK)$.
\end{proof}

\begin{corollary}\label{cor:R_T^G}
If $\bK \leqslant \Lang(\rZ(\bL))$ in the setting of
Proposition~\ref{prop:surj-isotyp-DL}, then ${}^{\top}\phi \circ R_{\bL'}^{\bG'}
= R_{\bL}^{\bG} \circ {}^{\top}\phi$. This condition is satisfied if either
$\bK$ is connected or $\rZ(\bL)$ is connected.
\end{corollary}

\begin{proof}
Under our assumption, we may choose $l_z \in \rZ(\bL)$ such that $\Lang(l_z) =
z$. In this case $\Ad_{l_z}|_{\bL}$ is trivial and the statement follows. Note
that $\bK \leqslant \rZ(\bG) \leqslant \rZ(\bL)$ and if $\bK$ is connected then
$\bK = \Lang(\bK) \leqslant \Lang(\rZ(\bL))$ and if $\rZ(\bL)$ is connected then
$\Lang(\rZ(\bL)) = \rZ(\bL)$.
\end{proof}

Note that Corollary~\ref{cor:R_T^G} applies in particular to the case where $\bL
= \rZ(\bL)$ is a torus. We give one further consequence of this formula applied
in the case of Harish-Chandra induction, see \cite[Prop.~12.1]{Bon06} for a
special case. This will be used in the next section.

\begin{proposition}\label{prop:cuspidal-pullback}
If $\phi : (\bG,F) \to (\bG',F')$ is an isotypy, then for any $\chi' \in
\irr(G')$ and $\chi \in \irr(G \mid \tw{\top}\phi(\chi'))$ we have $\chi$ is a
cuspidal character if and only if $\chi'$ is a cuspidal character.
\end{proposition}

\begin{proof}
The map $(\bL',\bP') \mapsto (\bL,\bP) = (\phi^{-1}(\bL'),\phi^{-1}(\bP'))$
gives a bijection between the pairs consisting of an $F'$-stable parabolic
subgroup $\bP' \leqslant \bG'$ with $F'$-stable Levi complement $\bL' \leqslant
\bP'$ and the corresponding set of pairs in $\bG$. If $\bK = \ker(\phi)$, then
given such pairs and $\psi' \in \irr(L')$, it follows from
Proposition~\ref{prop:surj-isotyp-DL} that
%%%%
\begin{equation}\label{eq:RLG-applied}
\abs{\bK/\Lang(\bK)}\tw{\top}\phi(R_{\bL' \subseteq \bP'}^{\bG'}(\psi'))
=
\sum_{z\Lang(\bK) \in
\bK/\Lang(\bK)}R_{\bL\subseteq\bP}^{\bG}(\tw{\top}\Ad_{l_z}(
{}^{\top}\phi(\psi'))).
\end{equation}
%%%%

The right hand side of \eqref{eq:RLG-applied} is a sum of characters. Therefore,
if $\chi' \in \irr(G' \mid R_{\bL'\subseteq\bP'}^{\bG'}(\psi'))$ then $\chi \in
\irr(G \mid R_{\bL \subseteq \bP}^{\bG}(\psi))$ for some $\psi \in \irr(L)$. So
if $\chi'$ is not cuspidal then neither is $\chi$. Conversely, suppose $\chi \in
\irr(G \mid R_{\bL\subseteq\bP}^{\bG}(\psi))$ for some $\psi \in \irr(L)$. If
$\hat\psi \in \irr(P)$ is the inflation of $\psi$ to $P = \bP^F$ then
$R_{\bL\subseteq \bP}^{\bG}(\psi) = \ind_P^G(\hat\psi)$. As $Z(G) \leqslant L
\leqslant P$ it follows from the induction formula that
%%%%
\begin{equation*}
\res_{Z(G)}^G(R_{\bL \subseteq \bP}^{\bG}(\psi))
=
[G:P]\psi(1)\omega_{\psi}
\end{equation*}
%%%%
where $\omega_{\psi} = \psi/\psi(1) \in \irr(\rZ(G))$. As $\res_{\rZ(G)}^G(\chi)
= \chi(1)\omega_{\chi}$ must occur in the left hand side, with $\omega_{\chi} = \chi/\chi(1) \in \irr(\rZ(G))$,
 we conclude that $\omega_{\psi} =
\omega_{\chi}$. In particular, $\psi$ has $K = \bK^F$ in its kernel because
$\chi$ does, so there exists a $\psi'' \in \irr(L')$ such that $\psi \in
\irr(L \mid \tw{\top}\phi(\psi''))$.

As $\chi$ must occur on the right hand side of \eqref{eq:RLG-applied}, with
$\psi'$ replaced by $\psi''$, and this is a sum of characters, it must also
occur in the left hand side. Therefore, there must exist a $\chi'' \in \irr(G'
\mid R_{\bL' \subseteq \bP'}^{\bG'}(\psi''))$ such that $\chi \in \irr(G \mid
\tw{\top}\phi(\chi''))$. By Theorem~\ref{thm:multfree}, restriction from $G'$ to
$\phi(G)$ is multiplicity free. It is then a consequence of Gallagher's
Theorem, and the fact that $G'/\phi(G)$ is abelian, that $\chi' = \chi''\lambda$
for some $\lambda \in \irr(G'/\phi(G))$.

Every $p$-element is in the kernel of $\lambda$ so $R_{\bL' \subseteq
\bP'}^{\bG'}(\psi'')\lambda = R_{\bL'\subseteq\bP'}^{\bG'}(\psi')$, where $\psi'
:= \psi''\res_{L'}^{G'}(\lambda)$, by \cite[Cor.~7.3.5]{dmbook2}. Therefore
$\chi' \in \irr(G' \mid R_{\bL' \subseteq \bP'}^{\bG'}(\psi'))$ so if $\chi$ is
not cuspidal then neither is $\chi'$.
\end{proof}

We investigate the implications Corollary~\ref{cor:R_T^G} has for Lusztig
series, following the arguments presented in \cite[Prop.~7.2]{Tay16}. First we
need to extend the discussion of dual isogenies to isotypies. For this, we
follow \cite[Def.~2.11]{Ruh22}.

\begin{definition}\label{def:dual-isotypy}
Assume $(\bG,F)$ and $(\bG',F')$ are finite reductive groups with dual groups
$(\bG^{\ast},F^{\ast})$ and $(\bG'^{\ast},F'^{\ast})$ with the dualities
witnessed by $(\bT_0,\bT_0^{\ast},\delta)$ and $(\bT_0',\bT_0'^{\ast},\delta')$
respectively. Two isotypies $\phi : (\bG,F) \to (\bG',F')$ and $\phi^{\ast} :
(\bG'^{\ast},F'^{\ast}) \to (\bG^{\ast},F^{\ast})$ are said to be \emph{dual} if
%%%%
\begin{equation}\label{eq:dual-isotypies}
\wc{X}(\phi^{\ast}\circ \Ad_{g^*}) \circ \delta' = \delta \circ X(\phi\circ \Ad_g)
\end{equation}
%%%%
for some $(g,g^*) \in \bG\times\bG'^{\ast}$ satisfying $\phi({}^g\bT_0)
\leqslant \bT_0'$ and $\phi^{\ast}({}^{g^*}\bT_0'^{\ast}) \leqslant
\bT_0^{\ast}$.
\end{definition}

Note the condition in \eqref{eq:dual-isotypies} generalises the condition in
\eqref{eq:commute-frob}. We need the following.

\begin{lemma}\label{lem:dual-isotypies}
Any isotypy $\phi : (\bG,F) \to (\bG',F')$ admits a dual $\phi^{\ast} :
(\bG'^{\ast},F'^{\ast}) \to (\bG^{\ast},F^{\ast})$, which is unique up to
composing with some $\Ad_h$ with $h \in \rN_{\bG'^*}(G'^*)$.
\end{lemma}

\begin{proof}
Note that $\phi(\bT_0)$ is a torus so is contained in a maximal torus of $\bG'$.
By the conjugacy of maximal tori there exists an element $g' \in \bG'$ such that
${}^{g'}\phi(\bT_0) \leqslant \bT_0'$. As $\bG' = \bG_{\der}'\cdot
\rZ^{\circ}(\bG')$ we can assume that $g' \in \bG_{\der}'$. Using (i) of
Lemma~\ref{lem:isotypy-der-diag} there is an element $g \in \bG_{\der}$ such
that $\phi(g) = g'$ and so $\phi({}^g\bT_0) \leqslant \bT_0'$.
The composition $\tilde\phi = \phi\circ \Ad_g$ is an isotypy $(\bG,\tilde F) \to
(\bG',F')$, where $\tilde F = \Ad_{\Lang(g)}\circ F$, which satisfies
$\tilde\phi(\bT_0) \leqslant \bT_0'$. Using the bijections stated in
Section~\ref{sec:DMmap} we see that we have a bijection
%%%%
\begin{equation*}
    {}^* : \Hom(\bT_0,\bT_0')
    \overset{\sim}{\longrightarrow}
    \Hom(\bT_0'^{\ast},\bT_0^{\ast}),
\end{equation*}
%%%%
which is defined by requiring that $X(f) =
\delta^{-1}\circ\wc{X}(f^*)\circ\delta'$.

If $\tilde f = \tilde\phi|_{\bT_0}$ then as $\tilde\phi$ is an isotypy $X(\tilde
f)$ defines a $p$-morphism of root data as defined in \cite[3.2]{Tay19}. It
follows that $\wc{X}(\tilde f^{\ast})$ will be a $p$-morphism of root data,
because $X(\tilde f)$ is, and so $X(\tilde f^{\ast})$ will be as well. By an
extension of the isogeny theorem, see \cite[Thm~3.8]{Tay19} and the references
therein, there exists an isotypy $\phi^{\ast} : \bG'^{\ast} \to \bG^{\ast}$ such
that $\phi^{\ast}(\bT_0'^{\ast}) \leqslant \bT_0^{\ast}$ and
$\phi^{\ast}|_{\bT_0'^{\ast}} = \tilde f^{\ast}$, which is then dual to $\phi$.

We now consider the unicity of $\phi^*$. If $h \in \rN_{\bG'^*}(G'^*)$ and $\psi
= \phi^* \circ \Ad_h$ then certainly $F^*\circ \psi = \psi \circ F'^*$ so $\psi$
is an isogeny $(\bG'^*,F'^*) \to (\bG^*,F^*)$. As $\psi \circ \Ad_{h^{-1}g^*} =
\phi^* \circ \Ad_{g^*}$ it follows that $\phi$ and $\psi$ are
dual.

Conversely, suppose $\phi$ and $\psi$ are dual and let $y \in \bG'^*$ be an
element such that $\psi({}^y\bT_0'^*) \leqslant \bT_0^*$. By the conjugacy of
maximal tori there exists an element $x \in \bG'^*$ such that ${}^{xy}\bT_0'^* =
{}^{g^*}\bT_0'^*$. Hence $\psi' = \psi\circ \Ad_{x^{-1}}$ satisfies
$\psi'({}^{g^*}\bT_0'^*) \leqslant \bT_0^*$. Because $\psi$ and $\phi^*$ both
satisfy \eqref{eq:dual-isotypies} we must have
%%%%
\begin{equation*}
\wc{X}(\psi' \circ \Ad_{xy}) 
= 
\wc{X}(\psi \circ \Ad_y) 
= 
\wc{X}(\phi^* \circ \Ad_{g^*}).
\end{equation*}
%%%%
This implies $\psi'\circ \Ad_{xyt} = \phi^* \circ \Ad_{g^*}$ for some $t \in
\bT_0'^*$, see \cite[Thm~3.8]{Tay19}. In particular, $\psi = \phi^*\circ\Ad_h$
for some $h \in \bG'^*$.

As $\psi$ and $\phi^*$ both commute with $F'^*$ and $F^*$, which are bijective,
we must have $\phi^*\Ad_{F'^*(h)h^{-1}} = \phi^*$. Hence, there exists a
homomorphism $\pi : \bG'^* \to \ker(\phi^*) \leqslant \rZ(\bG'^*)$ such that
$\Ad_{F'^*(h)h^{-1}}(x) = x\pi(x)$ for all $x \in \bG'^*$. However
$\bG_{\der}'^* \leqslant \ker(\pi)$, because $\ker(\phi^*)$ is abelian, and
$\rZ(\bG'^*)$ must also be in $\ker(\pi)$ by definition. Therefore $\pi$ is
trivial so $F'^*(h)h^{-1} \in \rZ(\bG'^*)$ and we may apply
Lemma~\ref{lem:normalizer-is-lang-preimage}.
\end{proof}

We can now give the analogue of \cite[Prop.~7.2]{Tay16} for arbitrary isotypies.

\begin{proposition}\label{prop:isotypyseries}
Assume $\phi : (\bG,F) \to (\bG',F')$ and $\phi^{\ast} : (\bG'^{\ast},F'^{\ast})
\to (\bG^{\ast},F^{\ast})$ are dual isotypies. If $\chi' \in
\mathcal{E}(G',s')$, for some semisimple $s' \in G'^{\ast}$, then $\irr(G \mid
{}^{\top}\phi(\chi')) \subseteq \mathcal{E}(G,\phi^{\ast}(s'))$.
\end{proposition}

\begin{proof}
Arguing as in the proof of \cite[Prop.~11.7(a)]{Bon06}, using the uniformity of
the regular character \cite[Cor.~10.2.6]{dmbook2}, we see that if $\chi \in
\irr(\bG^F)$ is an irreducible constituent of ${}^{\top}\phi(\chi')$ then $\chi$
occurs with non-zero multiplicity in some
$R_{\phi^{-1}(\bT')}^{\bG}({}^{\top}\phi(\theta'))$, where $(\bT',\theta')$ is
dual to some $(\bT'^{\ast},s')$.

We now just need to show that if $(\bT',\theta')$ corresponds to
$(\bT'^{\ast},s')$ then $(\bT,\theta)$ corresponds to
$(\bT^{\ast},\phi^{\ast}(s'))$. The argument here is exactly the same as that
given in the proof of \cite[Prop.~7.2]{Tay16}, which we note relies only on the
property in \eqref{eq:dual-isotypies}. 
\end{proof}

The $G'^*$-conjugacy class of $s'$ and ${}^hs'$ is the same for any $h \in
\rN_{\bG'^*}(G'^*)$. This is because $\rC_{\bG'^*}(s')$ contains a maximal torus
of $\bG'^*$, so $\rZ(\bG'^*)$ is contained in the connected component of
$\rC_{\bG'^*}(s')$. Hence, the series $\mathcal{E}(\bG^F,\phi^*(s'))$ is the
same regardless of which dual isotypy we pick by Lemma~\ref{lem:dual-isotypies}.

We now consider some consequences of these statements in our setting where
$\rC_{G^{\ast}}(s) \leqslant \rC_{\bG^{\star}}^{\circ}(s)$. Firstly, this
assumption is preserved by arbitrary isotypies.

\begin{lemma}\label{lem:Jaylemma}
Let $\phi : (\bG, F) \rightarrow (\bG', F')$ be an isotypy with dual isotypy
$\phi^{\ast} : (\bG'^{\ast},F'^{\ast}) \to (\bG^{\ast},F^{\ast})$. If $s' \in
G'^\ast$ is a semisimple element and $s = \phi^\ast(s')$ satisfies
$\rC_{G^{\ast}}(s) \leqslant \rC_{\bG^{\star}}^{\circ}(s)$, then
$\rC_{G'^{\ast}}(s') \leqslant \rC_{\bG'^{\star}}^{\circ}(s')$.
\end{lemma}

\begin{proof}
By \cite[Eq.~(2.2)]{bonnafequasiisol} we have
$\varphi^{\ast}(\rC_{\bG_1^{\ast}}^{\circ}(s_1))\cdot \rZ^{\circ}(\bG^{\ast}) =
\rC_{\bG^{\ast}}^{\circ}(s)$ and, arguing as above,
$\rC_{\bG_1^{\ast}}^{\circ}(s_1)$ contains $\ker(\varphi^{\ast})$.
\end{proof}

In the setting of Proposition~\ref{prop:isotypyseries} it is not necessarily the
case that ${}^{\top}\phi(\chi')$ is irreducible. We will show that this is the
case under our assumption that $\rC_{G^{\ast}}(s) \leqslant
\rC_{\bG^{\star}}^{\circ}(s)$. First we recall the case where $\phi$ is a
regular embedding \cite{Lu88}.

\begin{lemma}\label{restrictionbij}
Assume $\iota : (\bG,F) \to (\wt{\bG},F)$ is a regular embedding and $\tilde{s}\in
\wt{G}^{\ast}$ is a semisimple element. If $s = \iota^{\ast}(\tilde{s})$ satisfies
$\rC_{G^{\star}}(s) \leqslant \rC_{\bG^{\star}}^{\circ}(s)$ then
$\tw{\top}\iota$ induces a bijective map
$\cE(\wt{G},\tilde{s})\rightarrow\cE(G,s)$.
\end{lemma}

\begin{proof}
It follows from \cite[Prop.~5.1]{Lu88}, see also \cite[Prop.~11.5.2]{dmbook2},
that ${}^{\top}\iota(\tilde\chi)$ is irreducible for any $\tilde\chi \in
\cE(\wt{G},\tilde{s})$. Now $\cE(G,s) = \bigcup_{\tilde\chi \in \cE(\wt{G},\tilde{s})}
\irr(G\mid {}^{\top}\iota(\tilde\chi))$ by \cite[Prop.~11.7]{Bon06} and both
$\cE(G,s)$ and $\cE(\wt G,\tilde s)$ have the same cardinality by
\cite[Prop.~5.1]{Lu88}.
\end{proof}

\begin{lemma}\label{lem:transposephi}
Assume $\phi : (\bG,F) \to (\bG',F')$ and $\phi^{\ast} : (\bG'^{\ast},F'^{\ast})
\to (\bG^{\ast},F^{\ast})$ are dual isotypies and $s' \in G'^{\ast}$ is a
semisimple element. If $s = \phi^{\ast}(s')$ satisfies $\rC_{G^{\ast}}(s)
\leqslant \rC_{\bG^{\star}}^{\circ}(s)$ then $\tw{\top}\phi$ induces a bijective
map $\cE(G',s')\to\cE(G,s)$.
\end{lemma}

\begin{proof}
Note that $\phi(G) \leqslant G'$ may be a proper subgroup. Given
$\chi'\in\mathcal{E}(G',s')$, we have $\tw{\top}\phi(\chi')$ is the inflation to
$G$ of $\res^{G'}_{\phi(G)}(\chi')$. We claim that this restriction is
irreducible. Let $\chi_0$ be a constituent of $\res_{\phi(G)}^{G'}(\chi')$.
Since $\phi(G)$ contains
$\mathrm{O}^{p'}(G')=\mathrm{O}^{p'}(\phi(G))=\phi(\mathrm{O}^{p'}(G))$,
restrictions from $G'$ to $\phi(G)$ are multiplicity free by
Theorem~\ref{thm:multfree}. As $G'/\phi(G)$ is abelian, it suffices to show that
$(G')_{\chi_0}=G'$. 

Let $\chi = \tw{\top}\phi(\chi_0) \in\cE(G,s)$ be the inflation of $\chi_0$ and
fix a regular embedding $\iota : (\bG,F) \to (\wt{\bG},F)$. Since
$\rC_{G^{\ast}}(s) \leqslant \rC_{\bG^{\star}}^{\circ}(s)$, we have $\chi$
extends to $\wt{G}$ by Lemma~\ref{restrictionbij}, so $\wt{G}_\chi=\wt{G}$. By
Theorem~\ref{thm:multfree} and Lemma~\ref{prop:stabhatG}, we have $\chi$ also
extends to $\hat{G} = G\cdot \rN_{\bG_{\der}}(G)$ so $\chi_0$ extends to
$\phi(\hat G)$. Now, $G' \leqslant \phi(\hat{G})$, by (ii) of
Lemma~\ref{lem:overgroups-of-G} and (iv) Lemma~\ref{lem:isotypy-der-diag}, so
$(G')_{\chi_0} = G'$.

By Proposition~\ref{prop:isotypyseries} $\tw{\top}\phi(\chi')$ lies in
$\mathcal{E}(G,s)$. Moreover, $\rC_{G^{\star}}(s) =
\rC_{\bG^{\star}}^{\circ}(s)^{F^{\ast}}$ and by Lemma~\ref{lem:Jaylemma}
$\rC_{G'^{\star}}(s) = \rC_{\bG'^{\star}}^{\circ}(s')^{F^{\ast}}$ so
\cite[Prop.~11.3.8]{dmbook2} shows that $\tw{\top}\phi^*$ defines a bijection
$\cE(\rC_{G^{\ast}}(s),1) \to \cE(\rC_{G'^{\ast}}(s'),1)$ because
$\tw{\top}\phi^{\ast}$ restricts to an isotypy $\rC_{\bG^{\ast}}^{\circ}(s) \to
\rC_{\bG'^{\ast}}^{\circ}(s')$, see \cite[Eq.~(2.2)]{Bon06}. The Jordan
decomposition now shows that $\cE(G_1,s_1)$ and $\cE(G,s)$ have the same
cardinality.
\end{proof}

\section{The Jordan Decomposition and Isotypies}\label{sec:JD-and-isotypies}
In this section we investigate how the properties in Theorem~\ref{UniqueJord}
behave with respect to isotypies. With this in mind we fix, for the rest of this
section, an isotypy $\iota : (\bG,F) \to (\bG',F')$ and a dual isotypy
$\iota^{\ast} : (\bG'^{\ast},F'^{\ast}) \to (\bG^{\ast},F^{\ast})$ and we assume
that $s' \in G'^{\ast}$ is a semisimple element such that $s = \iota^{\ast}(s')
\in G^{\ast}$ satisfies $\rC_{G^{\ast}}(s) \leqslant
\rC_{\bG^{\ast}}^{\circ}(s)$. Note that the map
%%%%
\begin{equation*}
f \mapsto f_{\iota}
:= 
\tw{\top}\iota^{\ast} \circ f \circ \tw{\top}\iota
\end{equation*}
%%%%
identifies the set of bijections $\cE(G,s) \to \cE(\rC_{G^{\ast}}(s),1)$ with
the corresponding set $\cE(G',s') \to \cE(\rC_{G'^{\ast}}(s'),1)$ by
Lemma~\ref{lem:transposephi}. We warn the reader that whilst we omit
$\iota^{\ast}$ from the notation, the map $f_{\iota}$ does depend on it.

\begin{proposition}\label{Jordiff12}
We have 
$f : \cE(G,s) \to \cE(\rC_{G^{\ast}}(s),1)$ 
satisfies (1) of Theorem~\ref{UniqueJord} if and only if 
$f_{\iota} : \cE(G',s') \to \cE(\rC_{{G'}^{\ast}}(s'),1)$ 
does. Given $f$ satisfies (1) of Theorem~\ref{UniqueJord}, then $f$ satisfies (2) of Theorem~\ref{UniqueJord} if and only if $f_{\iota}$ does.  
\end{proposition}

\begin{proof}
Throughout we assume that $\chi' \in \cE(G', s')$ and $\chi \in \cE(G,s)$ are
characters such that $\chi = \tw{\top}\iota(\chi')$. Then $f_\iota(\chi')=\tw{\top}\iota^\ast( f(\chi))$. We consider (1) and (2)
separately.

(1). The map $\bT \mapsto \bT' := \iota(\bT)\cdot \rZ^{\circ}(\bG')$ is a
bijection, with inverse $\bT' \mapsto \iota^{-1}(\bT')$, between the $F$-stable
maximal tori of $\bG$ and those of $\bG'$. If $\theta \in \irr(T)$ and $\theta'
\in \irr(T')$ satisfy $\theta = {}^{\top}\iota(\theta')$ then $\langle
\chi,R_{\bT}^{\bG}(\theta)\rangle_G = \langle
\chi',R_{\bT'}^{\bG'}(\theta')\rangle_{G'}$ by Corollary~\ref{cor:R_T^G}, and similarly $
\langle f(\chi), \varepsilon_{\bG}\varepsilon_{\rC^{\circ}_{\bG^*}(s)}R_{\bT^\ast}^{\rC_{\bG^\ast}(s)}
(\triv) \rangle_{\rC_{G^\ast}^{\circ}(s)} = \langle f_{ \iota}(\chi'),
\varepsilon_{\bG'}\varepsilon_{\rC^{\circ}_{{\bG'}^*}(s)}R_{{\bT'}^\ast}^{\rC_{{\bG'}^\ast}(s')}
(\triv) \rangle_{\rC_{{G'}^\ast}^{\circ}(s')}$. The result follows. 

(2). We assume $s' = 1$ and $s=1$. Let $\psi = f(\chi) \in \cE(G^*, 1)$, and
$\psi' = \tw{\top}\iota^\ast(\psi) = f_\iota(\chi')$.  Consider Condition (2a).
The isotypies $\iota$ and $\iota^*$  naturally yield bijections of unipotent characters.  By
arguments in \cite[(1.18)]{Lu76}, together with \cite[Prop.~8.1.13]{dmbook2},
these bijections of unipotent characters from isotypies preserve the
corresponding eigenvalues of the Frobenius, so that the eigenvalues
corresponding to $\chi$ and $\chi'$ are equal, and the eigenvalues
corresponding to $\psi$ and $\psi'$ are equal.  The claim for Condition (2a)
follows.

For Property (2b), we now assume $\chi$ is in the principal series, 
from which it follows that so are $\chi'$, $\psi$, and $\psi'$, if we assume $f$, and thus also $f_{\iota}$, satisfy condition (1).

Let
${\mathbf B}$ be an $F$-stable Borel subgroup of $\bG$, and $B = {\mathbf B}^F$.
The Hecke algebra for $G$ 
%(or for $G^\ast$, $G'$, or ${G'}^\ast$) 
may be
described as $e\CC Ge$, with $e = \frac{1}{|B|} \sum_{b \in B} b$, and the
bijection from characters in the principal series $\ind_B^G(\triv)$ to characters of
$e \CC G e$ is given by extending $\chi$ from $G$ to $\CC G$ linearly, and
restricting to the subalgebra $e \CC Ge$.  The case for $G'$ is analogous, and thus the natural bijection between
the characters of Hecke algebras corresponding to $G$ and $G'$ is via composition with $\tw\top\iota$. Similarly,  the natural bijection between the characters of Hecke
algebras corresponding to $G^\ast$ and ${G'}^\ast$ is through composition with
$\tw\top\iota^*$, and then extending linearly. 
That is, $\chi$ and $\chi'$
correspond to the same character of the identified Hecke algebras through this
bijection, as do $\psi$ and $\psi'$.  
These identifications commute with the
canonical bijection between the Hecke algebras corresponding to $G$ and
$G^\ast$, which depends only on the underlying Weyl groups, the duality between
Weyl groups $W$ and $W^\ast$, and the canonical map of Lusztig between the Weyl
group and the Hecke algebra, see \cite{Lu81}.  The claim follows.
\end{proof}

\begin{proposition} \label{Jordiff3}
Assume $z' \in \rZ(G'^{\ast})$ is a central element and $z = \iota^{\ast}(z')
\in \rZ(G^{\ast})$. Then for any two bijections $f_s : \cE(G,s) \to
\cE(\rC_{G^{\ast}}(s),1)$ and $f_{sz} : \cE(G,sz) \to \cE(\rC_{G^{\ast}}(s),1)$
the following are equivalent:
%%%%
\begin{enumerate}
    \item $f_{sz}(\chi \otimes \hat{z}) = f_s(\chi)$ for all $\chi \in \cE(G,s)$
    \item $(f_{sz})_{\iota}(\chi' \otimes \hat{z'}) = (f_s)_{\iota}(\chi')$ for
    all $\chi' \in \cE(G',s')$.
\end{enumerate}
In particular, $f_s$ and $f_{sz}$ satisfy (3) of Theorem \ref{UniqueJord} if and only if $(f_s)_\iota$ and $(f_{sz})_\iota$ do.
%%%%
\end{proposition}

\begin{proof}
Note that
%%%%
\begin{equation*}
\rC_{G^{\ast}}(sz) = \rC_{G^{\ast}}(s)
\leqslant 
\rC_{\bG^{\ast}}^{\circ}(s) = \rC_{\bG^{\ast}}^{\circ}(sz)
\end{equation*}
%%%%
because $\rZ(G^{\ast}) \leqslant \rZ(\bG^{\ast})$, so ${}^{\top}\iota$ gives a
bijection $\cE(G',s'z') \to \cE(G,sz)$. Then we  use that  $\hat{z} =
\tw{\top}\iota(\hat{z'})$, by Proposition~\ref{prop:isotypyseries}.
\end{proof}

We now consider (4)--(5) of Theorem~\ref{UniqueJord}. The map $\bL \mapsto \bL'
= \iota(\bL)\cdot\rZ^{\circ}(\bG')$ gives a bijection between the Levi subgroups
of $\bG$ and those of $\bG'$ with inverse $\bL' \mapsto \iota^{-1}(\bL)$. We
assume $\bL$ is $F$-stable, which means $\bL'$ is $F'$-stable. Let $\bL^{\ast}
\leqslant \bG^{\ast}$ be an $F^{\ast}$-stable Levi subgroup dual to $\bL$ and
let $\bL'^{\ast} = (\iota^{\ast})^{-1}(\bL^{\ast})$.

If $\rC_{G^{\ast}}(s) \leqslant \rC_{\bG^\ast}^\circ(s) \leqslant \bL^\ast$ then
$\rC_{L^{\ast}}(s) \leqslant \bL^{\ast} \cap \rC_{\bG^\ast}^\circ(s) =
\rC_{\bL^\ast}^\circ(s)$, where the second equality follows from
\cite[1.4]{LS85}. Moreover, if $\rC_{\bG^\ast}^\circ(s) \leqslant \bL^\ast$ then
$\rC_{\bG'^\ast}^\circ(s') \leqslant \bL'^\ast$ by \cite[Prop.
2.3]{bonnafequasiisol}. As above the map
%%%%
\begin{equation*}
f^{\bL} \mapsto f_{\iota}^{\bL} 
:= 
(f^{\bL})_{\iota}
=
\tw{\top}\iota^{\ast} \circ f^{\bL} \circ \tw{\top}\iota
\end{equation*}
%%%%
identifies the set of bijections $\cE(L,s) \to \cE(\rC_{L^{\ast}}(s),1)$ with
the corresponding set $\cE(L',s') \to \cE(\rC_{L'^{\ast}}(s'),1)$.

\begin{proposition}\label{Jordiff45}
A pair of bijections $f^{\bG} : \cE(G,s) \to \cE(\rC_{G^{\ast}}(s),1)$ and
$f^{\bL} : \cE(L,s) \to \cE(\rC_{L^{\ast}}(s),1)$ satisfy (4), resp. (5), of
Theorem~\ref{UniqueJord} if and only if $f_{\iota}^{\bG}$ and $f_{\iota}^{\bL}$
do. 
\end{proposition}

\begin{proof}
The proof of Lemma~\ref{lem:transposephi} shows that if $l \in \bL$ and
$\Lang(l) \in \rZ(\bL)$ then for any $\chi \in \cE(L,s)$ we have
$\tw{\top}\Ad_l(\chi) = \chi$, since our conditions on $l$ ensure that
$\tw\top\Ad_l$ is an isotypy $\bL\rightarrow\bL$ that restricts to a map
$L\rightarrow L$. It follows from Proposition~\ref{prop:surj-isotyp-DL} that
%%%%
\begin{equation*}
f_{\iota}^{\bG}\circ R_{\bL'}^{\bG'}
= 
\tw{\top}\iota^{\ast} \circ f^{\bG} \circ R_{\bL}^{\bG} \circ \tw{\top}\iota
=
(f^{\bG}\circ R_{\bL}^{\bG})_{\iota}
\end{equation*}
%%%%
as maps $\ZZ\cE(L',s') \to \ZZ\cE(\rC_{G'^{\ast}}(s'),1)$. This proves the
statement for (4).

The same argument shows that
%%%%
\begin{equation*}
R_{\bL'^{\ast}}^{\rC_{\bG'^{\ast}}(s')} \circ f_{\iota}^{\bL}
= 
(R_{\bL^{\ast}}^{\rC_{\bG^{\ast}}(s)} \circ f^{\bL})_{\iota}
\end{equation*}
%%%%
as maps $\ZZ\cE(L',s') \to \ZZ\cE(\rC_{G'^{\ast}}(s'),1)$. By
Proposition~\ref{prop:cuspidal-pullback} the bijections $\cE(L',s') \to
\cE(L,s)$ and $\cE(\rC_{L^{\ast}}(s),1) \to \cE(\rC_{L'^{\ast}}(s'),1)$, induced
by $\tw{\top}\iota$ and $\tw{\top}\iota^{\ast}$ respectively, restrict to
bijections between the cuspidal parts of the series. Therefore $f_{\iota}^{\bL}$
restricts to a bijection $\cE(L',s')^{\bullet} \to
\cE(\rC_{L'^{\ast}}(s'),1)^{\bullet}$ between cuspidal characters if and only if
$f^{\bL}$ restricts to a bijection $\cE(L,s)^{\bullet} \to
\cE(\rC_{L^{\ast}}(s),1)^{\bullet}$.  From this the statement for (5) follows.
\end{proof}

We now want to investigate property (6), so we consider the situation where we
also have an isotypy $\iota_1 : (\bG_1,F_1) \to (\bG_1',F_1')$ as well as
isotypies $\varphi : (\bG,F) \to (\bG_1,F_1)$ and $\varphi' : (\bG',F') \to
(\bG_1',F_1')$ so that the following diagram commutes
%%%%
\begin{equation}\label{eq:comm-diag-isos}
\begin{tikzcd}[sep=1cm]
\bG' \arrow[r,"\varphi'"] & \bG_1'\\
\bG \arrow[r,"\varphi"]\arrow[u,"\iota"] & \bG_1\arrow[u,"\iota_1"']
\end{tikzcd}
\end{equation}
%%%%
As before we fix isotypies $\iota_1^{\ast}$, $\varphi^{\ast}$, and
$\varphi'^{\ast}$, that are dual to $\iota_1$, $\varphi$, and $\varphi'$,
respectively.

Both $\iota^{\ast}\circ\varphi'^{\ast}$ and $\varphi^{\ast}\circ\iota_1^{\ast}$
are dual to $\varphi'\circ\iota = \iota_1\circ\varphi$ so by
Lemma~\ref{lem:dual-isotypies} there is an element $g \in
\rN_{\bG^{\ast}}(G^{\ast})$ such that
%%%%
\begin{equation*}
\iota^{\ast}\circ\varphi'^{\ast} 
= 
\Ad_g\circ\varphi^{\ast}\circ\iota_1^{\ast}.
\end{equation*}
%%%%
Note that if $f : \cE(G,s) \to \cE(\rC_{G^{\ast}}(s),1)$ is a bijection then we
have $\tw{\top}\Ad_g\circ f$ is a bijection $\cE(G,s^g) \to
\cE(\rC_{G^{\ast}}(s^g),1)$ because $s^g$ and $s$ are $G^{\ast}$-conjugate.
Recall that we have picked a semisimple element $s' \in G'^{\ast}$
such that $s = \iota^{\ast}(s')$ satisfies $\rC_{G^{\ast}}(s) \leqslant
\rC_{\bG^{\ast}}^{\circ}(s)$.

\begin{proposition}\label{prop:commute-diag-isotypies}
Assume we have a commutative diagram as in \eqref{eq:comm-diag-isos} and let $g
\in \rN_{\bG^{\ast}}(G^{\ast})$ be an element such that
%%%%
\begin{equation*}
\iota^{\ast}\circ\varphi'^{\ast} 
= 
\Ad_g\circ\varphi^{\ast}\circ\iota_1^{\ast}.
\end{equation*}
%%%%
Furthermore, assume  $s_1' \in G_1'^{\ast}$ is a semisimple element such that
$s' = \varphi'^{\ast}(s_1')$ and let $s_1 = \iota_1^{\ast}(s_1')$. Then for any
two bijections $f_s^{\bG} : \cE(G,s) \to \cE(\rC_{G^{\star}}(s),1)$ and
$f_{s_1}^{\bG_1} : \cE(G_1,s_1) \to \cE(\rC_{G_1^{\star}}(s_1),1)$ the following
are equivalent:
%%%%
\begin{enumerate}
    \item $f_{s_1}^{\bG_1} = \tw{\top}\varphi^{\ast} \circ ({}^{\top}\Ad_g \circ
    f_s^{\bG}) \circ \tw{\top}\varphi$
    \item $(f_{s_1}^{\bG_1})_{\iota_1} = \tw{\top}\varphi'^{\ast} \circ
    (f_s^{\bG})_{\iota} \circ \tw{\top}\varphi'$
\end{enumerate}
\end{proposition}

\begin{proof}
For the convenience of the reader we record the maps under consideration in the
following diagram. We note, however, that not all squares of this diagram
commute.
%%%%
\begin{center}
\begin{tikzcd}[row sep = .75 cm, column sep = .5cm, math mode=true]
& 
\cE(G',s')
\arrow[rr,"(f_s^{\bG})_{\iota}"]
\arrow[dd,"\tw\top\iota",near start] 
& 
&
\cE(\rC_{G'^{\ast}}(s'),1) 
\arrow[dl,"{}^{\top}\varphi'^{\ast}"]\\
%%%%
\cE(G_1',s_1') 
\arrow[rr, "(f_{s_1}^{\bG_1})_{\iota_1}"', near start, crossing over]
\arrow[ur,"{}^{\top}\varphi'"]
\arrow[dd,"\tw{\top}\iota_1"'] 
& 
& 
\cE(\rC_{G_1'^{\ast}}(s_1'),1)\\
& 
\cE(G,s)
\arrow[rr,"f_s^{\bG}",near start] 
& 
& 
\cE(\rC_{G^{\ast}}(s),1) 
\arrow[dl,"{}^{\top}\varphi^{\ast}"]
\arrow[uu,"{}^{\top}\iota^{\ast}"]\\
%%%%%
\cE(G_1,s_1) 
\arrow[ur,"{}^{\top}\varphi"']
\arrow [rr,"f_{s_1}^{\bG_1}"'] 
& 
& 
\cE(\rC_{G_1^*}(s_1),1)
\arrow[uu,"{}^{\top}\iota_1^{\ast}",near start,crossing over]
\end{tikzcd}
\end{center}
%%%%
The front, back, and left squares of this diagram commute and the right square
commutes up to composing with $\tw{\top}\Ad_g$. A direct computation now shows
that
%%%%
\begin{align*}
\tw{\top}\varphi'^{\ast} \circ (f_s^{\bG})_{\iota} \circ \tw{\top}\varphi'
&=
\tw{\top}\varphi'^{\ast} \circ \tw{\top}\iota^{\ast} \circ f_s^{\bG} \circ
\tw{\top}\iota \circ \tw{\top}\varphi'\\
&=
\tw{\top}\iota_1^{\ast} \circ \tw{\top}\varphi^{\ast} 
\circ ({}^{\top}\Ad_g \circ f_s^{\bG}) \circ
\tw{\top}\varphi \circ \tw{\top}\iota_1.
\end{align*}
%%%%
Hence, up to replacing $f_s^{\bG}$ by $\tw{\top}\Ad_g \circ f_s^{\bG}$ we can
identify the top and bottom squares and thus the statement follows.
\end{proof}

\section{Existence and Unicity of the Jordan Decomposition}\label{sec:fixedregemb}
Our Jordan decomposition will be built in terms of the unique Jordan decomposition
defined by Digne--Michel in \cite[Thm.~7.1]{DiMi90}. To make effective use of
this in the more general setting of groups with a disconnected center, we will
need to strengthen (vi) of \cite[Thm.~7.1]{DiMi90} to include all isotypies. 
First, we record the following technical statement on dual groups and isotypies,
which is similar to the topics discussed in the proof of
\cite[Prop.~11.4.8]{dmbook2}.

\begin{lemma}\label{lem:dual-of-radical}
If $(\bG^{\ast},F^{\ast})$ is dual to $(\bG,F)$, then
$(\bG^{\ast}/(\bG^{\ast})_{\der},F^{\ast})$ is dual to $(\rZ^{\circ}(\bG),F)$.
Moreover, the isotypy $\pi^{\ast} : \bG^{\ast} \to \bG^{\ast} \times
\bG^{\ast}/(\bG^{\ast})_{\der}$ defined by $\pi^{\ast}(g) =
(g,g(\bG^{\ast})_{\der})$ is dual to the product $\pi : \bG \times
\rZ^{\circ}(\bG) \to \bG$ defined by $\pi((g,z)) = gz$.
\end{lemma}

\begin{proof}
Choose a witness $(\bT,\bT^{\star},\delta)$ to the duality and let $Q \subseteq
X(\bT)$ be the subgroup generated by the roots $\Phi_{\bG}(\bT)$. By
\cite[Prop.~8.1.8]{Spr} we have $X(\rZ^{\circ}(\bG)) = X(\bT)/Q^{\top}$ where
$Q^{\top}/Q$ is the torsion subgroup of $X(\bT)/Q$. We have a short exact
sequence
%%%%
\begin{center}
\begin{tikzcd}
1 
\arrow[r]
& 
(\bT^{\ast})_{\der}
\arrow[r]
&
\bT^{\ast}
\arrow[r]
&
\bG^{\ast}/(\bG^{\ast})_{\der}
\arrow[r]
&
1
\end{tikzcd}
\end{center}
%%%%
where $(\bT^{\ast})_{\der} := \bT^{\ast} \cap (\bG^{\ast})_{\der}$ and thus a
short exact sequence
%%%%
\begin{center}
\begin{tikzcd}
1 
\arrow[r]
& 
X(\bG^{\ast}/(\bG^{\ast})_{\der})
\arrow[r]
&
X(\bT^{\ast})
\arrow[r]
&
X((\bT^{\ast})_{\der})
\arrow[r]
&
1.
\end{tikzcd}
\end{center}
%%%%

By \cite[Prop.~8.1.8]{Spr} $X((\bT^{\ast})_{\der}) \cong
X(\bT^{\ast})/(\wc{Q}^{\ast})^{\perp}$ where $\wc{Q}^{\ast} \subseteq
\wc{X}(\bT^{\ast})$ is the subgroup generated by the coroots
$\wc\Phi_{\bG^{\ast}}(\bT^{\ast})$ and
%%%%
\begin{equation*}
(\wc{Q}^{\ast})^{\perp} 
= 
\{x \in X(\bT^{\ast}) \mid \langle x,y\rangle = 0\text{ for all }y \in
\wc{Q}^{\ast}\}.
\end{equation*}
%%%%
Therefore $X(\bG^{\ast}/(\bG^{\ast})_{\der}) = (\wc Q^{\ast})^{\perp}$ so
$\wc{X}(\bG^{\ast}/(\bG^{\ast})_{\der}) =
\wc{X}(\bT^{\ast})/(\wc{Q}^{\ast})^{\top}$ because $(\wc{Q}^{\ast})^{\perp\perp}
= (\wc{Q}^{\ast})^{\top}$. We get the first statement because $\delta(Q)  = \wc
Q^{\ast}$ and so $\delta(Q^{\top}) = (\wc Q^{\ast})^{\top}$.

The map $\delta' = \delta\oplus\delta : X(\bT)\oplus X(\rZ^{\circ}(\bG)) \to
\wc{X}(\bT^{\ast}) \oplus \wc{X}(\bG^{\ast}/(\bG^{\ast})_{\der})$ is an
isomorphism of root data. Further,
%%%%
\begin{equation*}
X(\pi)
: 
X(\bT)
\to
X(\bT) \oplus X(\rZ^{\circ}(\bG))
\end{equation*}
%%%%
is given by $\chi \mapsto (\chi,\chi + Q^{\top})$ and from this we deduce that
$\wc{X}(\pi^*)\circ \delta = \delta' \circ X(\pi)$.
\end{proof}

We are now ready to prove one of our main results.

\begin{theorem}\label{thm:prop-6-DM}
Assume $\phi : (\bG,F) \to (\bG',F')$ is an isotypy between finite reductive
groups with both $\rZ(\bG)$ and $\rZ(\bG')$ connected and let $\phi^{\ast} :
(\bG'^{\ast},F^{\ast}) \to (\bG^{\ast},F^{\ast})$ be an isotypy dual to $\phi$.
If $s' \in G'^{\ast}$ is a semisimple element and $s = \phi^{\ast}(s')$, then
%%%%
\begin{equation*}
J_{s'}^{\bG'} = \tw{\top}\phi^{\ast} \circ J_s^{\bG} \circ \tw{\top}\phi,
\end{equation*}
%%%%
where $J_s^{\bG} : \cE(G,s) \to \cE(\rC_{G^{\ast}}(s),1)$ and $J_{s'}^{\bG'} :
\cE(G',s') \to \cE(\rC_{G'^{\ast}}(s'),1)$ are the unique Jordan decompositions
defined by Digne--Michel in \cite{DiMi90}.
\end{theorem}

\begin{proof}
Let $\wt\bG = \bG\times\rZ(\bG')$ and $\wt\bG' = \bG'\times \rZ(\bG')$ be
equipped with the  Frobenius endomorphisms $ F\times F' : \wt\bG \to
\wt\bG$ and $ F'\times F' : \wt\bG' \to \wt\bG'$, respectively, which we again denote by $F$ and $F'$. We denote by
$\pi : \wt\bG' \to \bG'$ the natural product morphism $(g,z) \mapsto gz$. Note
that the group $\rZ(\wt\bG) = \rZ(\bG) \times \rZ(\bG')$ is connected. By
Lemma~\ref{lem:dual-of-radical} we have $\wt\bG^{\ast} = \bG^{\ast} \times
\bG'^{\ast}/(\bG'^{\ast})_{\der}$ and $\wt\bG'^{\ast} = \bG'^{\ast} \times
\bG'^{\ast}/(\bG'^{\ast})_{\der}$ are dual groups of $\wt\bG$ and $\wt\bG'$
respectively. Moreover, $\pi^{\ast} : \bG'^{\ast} \to \bG'^{\ast} \times
\bG'^{\ast}/(\bG'^{\ast})_{\der}$, defined by $g \mapsto
(g,g(\bG'^{\ast})_{\der})$, is dual to $\pi$.

Let $\tilde\phi = \pi \circ (\phi\times \mathrm{Id}) : \wt\bG \to \bG'$ be
defined by $\tilde\phi(g,z) = \phi(g)z$. Then by (ii) of
Lemma~\ref{lem:isotypy-der-diag}, we see the kernel
%%%%
\begin{equation*}
\ker(\tilde{\phi}) = \{(z,\phi(z)^{-1}) \mid z \in \rZ(\bG)\} \cong \rZ(\bG)
\end{equation*}
%%%%
is a central torus. As the kernel of $\tilde{\phi}$ is connected this restricts
to an epimorphism $\wt G \to G'$ at the level of rational points, see
\cite[Lem.~4.2.13]{dmbook2}, so $G' = \phi(G)\rZ(G')$. In other words, $G'$ is a
central product of $\phi(G)$ and $\rZ(G')$. Let $\tilde\phi' = \pi \circ
(\phi\times 1) : \wt\bG \to \bG'$ be the isotypy defined by $\tilde\phi'((g,z))
= \phi(g)$. If $\chi \in \irr(G')$ has central character $\omega_{\chi} =
\chi/\chi(1) \in \irr(\rZ(G'))$ and $\tilde{g} = (g,z) \in \wt G$ then
%%%%
\begin{equation*}
\tw{\top}\tilde{\phi}(\chi)(\tilde{g})
= 
\chi(\phi(g)z)
= 
\chi(\phi(g))\omega_{\chi}(z)
=
\tw{\top}\tilde{\phi}'(\chi)(\tilde{g})\theta(\tilde{g})
\end{equation*}
%%%%
where $\theta = 1\boxtimes \omega_{\chi} \in \irr(\wt G)$. In other words,
$\tw{\top}\tilde\phi(\chi) = \tw{\top}\tilde\phi'(\chi)\theta$. As $\omega_{\chi}$ is
constant on Lusztig series, see \cite[Prop.~9.11]{Bon06}, we see that
$\tw{\top}\tilde\phi(\chi) = \tw{\top}\tilde\phi'(\chi)\theta$ for all $\chi \in
\cE(G',s')$.

The isotypies $(\bG'^{\ast},F'^{\ast}) \to (\wt\bG^{\ast},F^{\ast})$ defined by
$\tilde{\phi}^{\ast} = (\phi^{\ast} \times \mathrm{Id})\circ\pi^{\ast}$ and 
$\tilde{\phi}'^{\ast} = (\phi^{\ast} \times 1)\circ\pi^{\ast}$, meaning
%%%%
\begin{align*}
\tilde\phi^{\ast}(g) &= (\phi^{\ast}(g),g(\bG'^{\ast})_{\der}),\\
\tilde\phi'^{\ast}(g) &= (\phi^{\ast}(g),(\bG'^{\ast})_{\der}),
\end{align*}
%%%%
for all $g \in \bG'^{\ast}$, are dual to $\tilde\phi$ and $\tilde\phi'$
respectively. Setting $\tilde{s} := (s,(\bG'^{\ast})_{\der})$ and $z :=
(1,s'(\bG'^{\ast})_{\der}) \in \rZ(\wt{\bG}^{\ast})$ we see that
$\tilde\phi^{\ast}(s') = \tilde sz$ and $\theta = \hat{z}$. Note that
$\tw{\top}\tilde\phi(\chi) = \tw{\top}\tilde\phi'(\chi)\theta \in \cE(\wt
G,\tilde sz)$ and $\tw{\top}\tilde\phi'(\chi) \in \cE(\wt G,\tilde s)$ by
Proposition~\ref{prop:isotypyseries}. By (iii) of \cite[Thm.~7.1]{DiMi90}
%%%%
\begin{equation*}
J_{\tilde sz}^{\wt\bG}\circ\tw{\top}\tilde\phi
=
J_{\tilde s}^{\wt\bG}\circ\tw{\top}\tilde\phi'.
\end{equation*}

On the other hand, $\rC_{\wt G^{\ast}}(\tilde s) = \rC_{\wt G^{\ast}}(\tilde sz)
= \rC_{G^{\ast}}(s) \times (\bG'^{\ast}/(\bG'^{\ast})_{\der})^{F'}$ and the maps
$\tw{\top}\tilde\phi^{\ast}$ and $\tw{\top}\tilde\phi'^{\ast}$ define the same
bijection $\cE(\rC_{\wt G^{\ast}}(\tilde s),1) \to \cE(\rC_{G'^{\ast}}(s'),1)$
because every unipotent character of $\rC_{\wt G^{\ast}}(\tilde s)$ is of the
form $\psi \boxtimes \triv$ for some $\psi \in \cE(\rC_{G^{\ast}}(s),1)$.
Therefore,
%%%%
\begin{equation*}
J_{s'}^{\bG'}
=
\tw{\top}\tilde\phi^{\ast} 
    \circ J_{\tilde sz}^{\wt\bG} 
    \circ \tw{\top}\tilde\phi
=
\tw{\top}\tilde\phi'^{\ast}
    \circ J_{\tilde s}^{\wt\bG}
    \circ \tw{\top}\tilde\phi'.
\end{equation*}
%%%%
by (vi) of \cite[Thm.~7.1]{DiMi90}.

Now let $\pi_1 : (\wt\bG,F) \to (\bG,F)$ be the natural projection onto the
first factor. This is an epimorphism with kernel $\ker(\pi) = 1\times \rZ(\bG')$
a central torus. The map $\pi_1^{\ast} : (\bG^{\ast},F^{\ast}) \to
(\wt\bG^{\ast},F^{\ast})$ defined by $\pi_1^{\ast}(g) =
(g,(\bG'^{\ast})_{\der})$ is dual to $\pi_1$, and another
application of (vi) of \cite[Thm.~7.1]{DiMi90} gives us 
%%%%
\begin{equation*}
J_s^{\bG}
=
\tw{\top}\pi_1^{\ast} 
    \circ J_{\pi_1^{\ast}(s)}^{\wt\bG} 
    \circ \tw{\top}\pi_1
=
\tw{\top}\pi_1^{\ast} 
    \circ J_{\tilde s}^{\wt\bG} 
    \circ \tw{\top}\pi_1.
\end{equation*}
%%%%
Using that $\tilde\phi' = \phi\circ\pi_1$ we find
%%%%
\begin{equation*}
\tw{\top}\phi^{\ast} \circ J_s^{\bG} \circ \tw{\top}\phi
= 
\tw{\top}\tilde\phi'^{\ast} 
    \circ J_{\tilde{s}}^{\wt\bG}
    \circ \tw{\top}\tilde\phi',
\end{equation*}
%%%%
which completes the proof.
\end{proof}

With this in hand, we can now define the bijections of Theorem~\ref{UniqueJord}.
For the rest of this section, we assume that
%%%%
\begin{equation*}
\cD = ((\bG,F),(\bG^{\ast},F^{\ast}),(\bT_0,\bT_0^{\ast},\delta_0))
\end{equation*}
%%%%
is a rational duality. Let $\iota : (\bG,F) \to (\wt\bG,F)$ be a regular
embedding and $\iota^{\ast} : (\wt\bG^{\ast},F^{\ast}) \to
(\bG^{\ast},F^{\ast})$ a dual isotypy. Recall from \cite[Cor.~2.6]{Bon06} that
$\iota^{\ast}(\wt G^{\ast}) = G^{\ast}$, in other words, $\iota^{\ast}$
restricts to an epimorphism $\wt G^{\ast} \twoheadrightarrow G^{\ast}$ between
finite groups.

From now on, $s \in G^{\ast}$ is a semisimple element with $\rC_{G^{\ast}}(s)
\leqslant \rC_{\bG^{\ast}}^{\circ}(s)$ and $\tilde s \in \wt G^{\ast}$ is an
element such that $s = \iota^{\ast}(\tilde s)$. Then using Lemma \ref{restrictionbij}, there exists a unique bijection
$J_{s,\iota}^{\cD} : \cE(G,s) \to \cE(\rC_{G^{\ast}}(s),1)$ such that
%%%%
\begin{equation}\label{eq:def-Js}
J_{\tilde s}^{\wt\cD} 
:= \tw{\top}\iota^{\ast} \circ J_{s,\iota}^{\cD} \circ \tw{\top}\iota
: 
\cE(\wt G,\tilde s) \to \cE(\rC_{\wt G^{\ast}}(\tilde s),1)
\end{equation}
%%%%
is Digne--Michel's unique Jordan decomposition, as defined by
\cite[Thm.~7.1]{DiMi90}. As in Section~\ref{sec:DMmap} we will denote $J_{\tilde
s}^{\wt \cD}$ and $J_{s,\iota}^{\cD}$ by $J_{\tilde s}^{\wt \bG}$ and
$J_{s,\iota}^{\bG}$, respectively. We will need some basic properties of
$J_s^{\cD}$ that follow from the results in \cite{DiMi90}.

\begin{proposition}\label{prop:bumper-Js}
\mbox{}
%%%%
\begin{enumerate}
    \item If $\tilde{s}' \in \wt G^{\ast}$ is another element such that $s =
\iota^{\ast}(\tilde{s}')$ then $J_{\tilde s'}^{\wt\bG} = \tw{\top}\iota^{\ast}
\circ J_{s,\iota}^{\bG} \circ \tw{\top}\iota$ as bijections $\cE(\wt G,\tilde
s') \to \cE(\rC_{\wt G^*}(\tilde s),1)$. In particular, $J_{s,\iota}^{\bG}$ is
independent of the choice of element $\tilde s$ used in its definition.

    \item If $g \in \rN_{\bG^{\ast}}(G^{\ast})$ then $J_{s^g,\iota}^{\bG} =
    \tw{\top}\Ad_g\circ J_{s,\iota}^{\bG}$.
    \item If $\iota_i : (\bG,F) \to (\wt\bG_i,F)$ are regular embeddings, with
    $i \in \{1,2\}$, then $J_{s,\iota_1}^{\bG} = J_{s,\iota_2}^{\bG}$.
\end{enumerate}
\end{proposition}

\begin{proof}
(1). By the above remark we have $\tilde{s}' = \tilde{s}z$ for some $z \in
\ker(\iota^{\ast})^{F^{\ast}} \leqslant \rZ(\wt G^{\ast})$. The characters of
$\mathcal{E}(\wt{G}, \tilde{s}')$ are just $\tilde{\chi}\otimes \wh{z}$ for
$\tilde{\chi}\in\mathcal{E}(\wt{G}, \tilde{s})$. Hence it suffices to note that
$\tw{\top}\iota(\tilde{\chi}\otimes \wh{z})=\tw{\top}\iota(\tilde{\chi})$ and
$\rC_{\wt{G}^\ast}(\tilde{s})=\rC_{\wt{G}^\ast}(\tilde{s}z)$ and apply property (iii)
of \cite[Thm.~7.1]{DiMi90}.

(2). Pick an element $\tilde g \in \wt G^{\ast}$ such that $\iota^{\ast}(\tilde g) =
g$. Note that we have
%%%%
\begin{equation*}
\tw{\top}\Ad_{\tilde g} \circ J_{\tilde s}^{\wt \bG}
=
\tw{\top}\Ad_{\tilde g} \circ \tw{\top}\iota^{\ast} 
    \circ J_{s,\iota}^{\bG} \circ \tw{\top}\iota
=
\tw{\top}\iota^{\ast} \circ \tw{\top}\Ad_g 
    \circ J_{s,\iota}^{\bG} \circ \tw{\top}\iota.
\end{equation*}
%%%%
By \cite[Cor.~7.3]{DiMi90}, suitably reformulated in our setting, we have the
left hand side is $J_{\tilde t}^{\wt \bG}$, where $\tilde t = \tilde s^{\tilde g}$. As
$\iota^{\ast}(\tilde t) = s^g$ the statement follows from (1).

(3). By a result of Asai, see \cite[Thm.~1.19]{Tay19}, there exist regular
embeddings $\iota_i' : (\wt\bG_i,F) \to (\bG',F)$ such that the following
diagram is commutative
%%%%
\begin{center}
\begin{tikzcd}[sep=1cm]
\bG
\arrow[r,"\iota_1"]
\arrow[d,swap,"\iota_2"]
&
\wt\bG_1
\arrow[d,"\iota_1'"]
\\
\wt\bG_2
\arrow[r,"\iota_2'"]
&
\bG'
\end{tikzcd}
\end{center}
%%%%
Assume we have fixed dual isotypies $\iota_i'^{\ast} : (\bG'^{\ast},F^{\ast})
\to (\bG_i^{\ast},F^{\ast})$. By Lemma \ref{lem:dual-isotypies}, there exists an element $g \in
\rN_{\bG'^{\ast}}(G'^{\ast})$ such that
%%%%
\begin{equation*}
\iota_1^{\ast} \circ \iota_1'^{\ast}
=
\iota_2^{\ast} \circ \iota_2'^{\ast} \circ \Ad_g.
\end{equation*}
%%%%
Pick an element $s_1' \in G'^{\ast}$ such that $s =
\iota_1^{\ast}(\iota_1'^{\ast}(s_1'))$ and let $\tilde s_1 =
\iota_1'^{\ast}(s_1')$. By Theorem~\ref{thm:prop-6-DM}, we have
%%%%
\begin{align*}
J_{s_1'}^{\bG'}
&=
\tw{\top}\iota_1'^{\ast}
    \circ J_{\tilde s_1}^{\wt\bG_1}
    \circ \tw{\top}\iota_1'
\\
&= 
\tw{\top}\iota_1'^{\ast}
    \circ \tw{\top}\iota_1^{\ast} 
    \circ J_{s,\iota_1}^{\bG} 
    \circ \tw{\top}\iota_1
    \circ \tw{\top}\iota_1'
\\
&= 
\tw{\top}\Ad_g
    \circ \tw{\top}\iota_2'^{\ast}
    \circ \tw{\top}\iota_2^{\ast} 
    \circ J_{s,\iota_1}^{\bG} 
    \circ \tw{\top}\iota_2
    \circ \tw{\top}\iota_2'.
\end{align*}

On the other hand, if $s_2' = {}^gs_1'$ and $\tilde{s}_2 =
\iota_2'^{\ast}(s_2')$, then $s = \iota_2^{\ast}(\tilde{s}_2)$ and by (2) and
Theorem~\ref{thm:prop-6-DM}, we see that
%%%%
\begin{equation*}
J_{s_1'}^{\bG'}
=
\tw{\top}\Ad_g
    \circ J_{s_2'}^{\bG'}
= 
\tw{\top}\Ad_g
    \circ \tw{\top}\iota_2'^{\ast}
    \circ J_{\tilde s_2}^{\wt \bG_2}
    \circ \tw{\top}\iota_2'.
\end{equation*}
%%%%
Comparing these two equalities shows that $J_{\tilde{s}_2}^{\wt\bG_2} =
\tw{\top}\iota_2^{\ast} \circ J_{s,\iota_1}^{\bG} \circ \tw{\top}\iota_2$, which
proves the statement.
\end{proof}

In light of Proposition~\ref{prop:bumper-Js} we will denote the bijection
$J_{s,\iota}^{\cD}$ defined in \eqref{eq:def-Js} simply by $J_s^{\cD}$ or
$J_s^{\bG}$. Let us start by pointing out that $J_s^{\bG}$ is a
$\gal$-equivariant Jordan decomposition map, which will be sufficient for many
applications. By the proof of \cite[Lemma 3.4]{SFT18}, if $\sigma \in \gal$ then
there is some $s^\sigma\in G^\ast$ such that
$\mathcal{E}(G,s)^\sigma=\mathcal{E}(G,s^\sigma)$. Namely, if $\sigma$ maps
$|s|$th roots of unity to their $k$th power with $(|s|,k)=1$, then
$s^\sigma=s^k$. Of course, $\rC_{\bG^{\star}}(s^\sigma) = \rC_{\bG^{\star}}(s)$
and with this we extend the main result of \cite{SrVi19} to our setting.

\begin{lemma}\label{galequiv}
The bijection $J_s^{\bG} : \cE(G,s) \to \cE(\rC_{G^{\ast}}(s),1)$ is a Jordan
decomposition satisfying $J_{s^\sigma}^{\bG}(\chi^\sigma) =
J_s^{\bG}(\chi)^\sigma$ for any $\chi \in \cE(G,s)$ and $\sigma \in \gal$.
\end{lemma}

\begin{proof}
By Proposition~\ref{Jordiff12}, we have $J_s^{\bG}$ is a Jordan decomposition.
If $\tilde\chi \in \cE(\wt G,\tilde s)$, then
$J_{\tilde{s}^\sigma}^{\wt\bG}(\tilde{\chi}^\sigma) =
J_{\tilde{s}}^{\wt\bG}(\tilde{\chi})^\sigma$ by \cite{SrVi19}. As
$\tw{\top}\iota^{\ast}$ and $\tw{\top}\iota$ commute with the Galois action,
$J_{s^\sigma}^{\bG}(\tw{\top}\iota(\tilde{\chi})^\sigma) =
J_s^{\bG}(\tw{\top}\iota(\tilde{\chi}))^\sigma$.
\end{proof}

We are now ready to prove that our bijection $J_s^{\bG}$ satisfies the
properties listed in Theorem~\ref{UniqueJord}. For this we will need the
following classical construction of regular embeddings, see \cite[1.21]{DL76}.

\begin{definition}\label{def:iotaspecific}
Let $(\bG,F)$ be a finite reductive group with fixed $F$-stable maximal torus
$\bT \leqslant \bG$ and define $\bG\times_{\rZ(\bG)}\bT = (\bG\times
\bT)/\Delta$, where
%%%%
\begin{equation*}
\Delta = \Delta(\bG,\bT) = \{(z,z^{-1})\mid z \in \rZ(\bG)\}.
\end{equation*}
%%%%
We have $\bG\times_{\rZ(\bG)}\bT$ inherits a Frobenius endomorphism, defined by
$F((g,t)\Delta) = (F(g),F(t))\Delta$ and the natural map $\iota_{\bT} : (\bG,F)
\to (\bG\times_{\rZ(\bG)}\bT,F)$, given by $g \mapsto (g,1)\Delta$, is a regular
embedding.
\end{definition}

If $\bT' \leqslant \bG$ is another $F$-stable maximal torus of $\bG$ then $\bT'
= {}^g\bT$ for some $g \in \bG$. It is not difficult to check that there is an
isomorphism $\bG \times_{\rZ(\bG)}\bT \to \bG\times_{\rZ(\bG)} \bT'$ given by
$(x,t)\Delta(\bG,\bT) \mapsto (x,{}^gt)\Delta(\bG,\bT')$. However, this
isomorphism does not commute with the induced Frobenius endomorphisms in
general. These regular embeddings have the following desirable lifting
properties.

\begin{lemma}\label{lem:epilifts}
Suppose $\phi : (\bG,F) \to (\bG_1,F_1)$ is an isotypy and $\bT \leqslant \bG$
and $\bT_1 \leqslant \bG_1$ are $F$-stable and $F_1$-stable maximal tori
satisfying $\phi(\bT) \leqslant \bT_1$. Then there exists an isotypy 
%%%%
\begin{equation*}
\tilde{\phi}
: (\bG\times_{\rZ(\bG)} \bT,F) \to (\bG_1\times_{\rZ(\bG_1)}\bT_1,F_1)
\end{equation*}
%%%%
satisfying $\iota_{\bT_1}\circ\phi = \tilde\phi \circ \iota_{\bT}$.
\end{lemma}

\begin{proof}
We have a homomorphism $\pi : \bG \times \bT \to \bG_1 \times_{\rZ(\bG_1)}\bT_1$
defined by $\pi((x,t)) = (\phi(x),\phi(t))\Delta_1$ where $\Delta_1 =
\Delta(\bG_1,\bT_1)$. Certainly $(x,t) \in \ker(\pi)$ if and only if $\phi(x) =
\phi(t^{-1}) \in \rZ(\bG_1)$ so by Lemma~\ref{lem:isotypy-der-diag} $\ker(\pi) =
\bK\Delta$ where $\Delta = \Delta(\bG,\bT)$ and $\bK = \{(z,1) \mid z \in
\ker(\phi)\}$. It follows that $\pi$ factors through a homomorphism $\tilde{\phi} :
\wt{\bG} \to \wt{\bG}_1$ satisfying $\iota_{\bT_1}\circ\phi = \tilde\phi \circ
\iota_{\bT}$.
\end{proof}

With this we can now prove Theorem~\ref{UniqueJord}.

\begin{proof}[Proof of Theorem~\ref{UniqueJord}]
We start by showing that the bijections $J_s^{\bG}: \cE(G,s) \to
\cE(\rC_{G^{\ast}}(s),1)$ satisfy properties (1)--(7) of
Theorem~\ref{UniqueJord}. This will give the existence part of the theorem. From the existence of such a Jordan decomposition map in the case of connected center by \cite[Thm.~7.1]{DiMi90}, along with Propositions \ref{Jordiff12}, \ref{Jordiff3}, and \ref{Jordiff45}, it follows that $J_s^{\bG}$ satisfies the conditions
(1)--(5). We now show $J_s^{\bG}$ satisfies (6). For this we assume that
$\varphi : (\bG,F) \to (\bG_1,F_1)$ is an isotypy. Choose $F$-stable and
$F_1$-stable maximal tori $\bT \leqslant \bG$ and $\bT_1 \leqslant \bG_1$ such
that $\varphi(\bT) \leqslant \bT_1$ and let $\wt\bG = \bG \times_{\rZ(\bG)} \bT$
and $\wt\bG_1 = \bG_1\times_{\rZ(\bG)} \bT_1$. By Lemma~\ref{lem:epilifts}, we
can find an isotypy $\tilde\varphi : (\wt\bG,F) \to (\wt\bG_1,F_1)$ satisfying
$\iota_{\bT_1}\circ\varphi = \tilde\varphi \circ\iota_{\bT}$.

Pick isotypies $\varphi^{\ast}$, $\tilde\varphi^{\ast}$, $\iota_{\bT}^{\ast}$,
and $\iota_{\bT_1}^{\ast}$, dual to $\varphi$, $\tilde\varphi$, $\iota_{\bT}$,
and $\iota_{\bT_1}$, respectively. By assumption there exists an element $s_1
\in G_1^{\ast}$ satisfying $s = \varphi^{\ast}(s_1)$. We choose an element
$\tilde{s}_1 \in \wt G_1^{\ast}$ satisfying $s_1 =
\iota_{\bT_1}^{\ast}(\tilde{s}_1)$ and let $\tilde{s} =
\tilde\varphi^{\ast}(\tilde{s}_1) \in \wt G^{\ast}$. By definition,
$J_{\tilde{s}}^{\wt\bG} = \tw{\top}\iota_{\bT}^{\ast} \circ J_s^{\bG} \circ
\tw{\top}\iota_{\bT}$ and $J_{\tilde{s}_1}^{\wt\bG_1} =
\tw{\top}\iota_{\bT_1}^{\ast} \circ J_{s_1}^{\bG} \circ \tw{\top}\iota_{\bT_1}$. By Theorem~\ref{thm:prop-6-DM}, we have $J_{\tilde{s}_1}^{\wt\bG_1} =
\tw{\top}\tilde\varphi^{\ast} \circ J_{\tilde{s}}^{\bG} \circ \tw{\top}\varphi$.
It follows from (2) of Proposition~\ref{prop:bumper-Js} and
Proposition~\ref{prop:commute-diag-isotypies} that $J_{s_1}^{\bG_1} =
\tw{\top}\varphi^{\ast} \circ J_s^{\bG} \circ \tw{\top}\varphi$.

Finally, we show $J_s^{\bG}$ satisfies (7). Assume $\bG=\prod_j \bG_j$ is a
direct product. We fix a regular embedding $\iota_j : \bG_j \to \wt\bG_j$ for
each $j$ and a dual $\iota_j^{\ast} : \wt\bG_j^{\ast} \to \bG_j^{\ast}$. If
$\wt\bG = \prod_j \wt\bG_j$ and $\wt\bG^{\ast} = \prod_j \wt\bG_j^{\ast}$ then
$\iota = \prod_j \iota_j : \bG \to \wt\bG$ is a regular embedding of $\bG$ with
dual $\iota^{\ast} = \prod \iota_j^{\ast} : \wt\bG^{\ast} \to \bG^{\ast}$. We
have $s = \prod_j s_j$ and $\tilde{s} = \prod_j \tilde{s}_j$. By
\cite[Thm.~7.1]{DiMi90}, and the definition of our bijection,
%%%%
\begin{equation*}
\tw{\top}\iota^{\ast} \circ J_s^{\bG} \circ \tw{\top}\iota
=
J_{\tilde{s}}^{\tilde\bG} 
= 
\prod_{j} J_{\tilde{s}_j}^{\wt\bG_j}
=
\prod_j \tw{\top}\iota_j^{\ast} \circ J_{s_j}^{\bG_j} \circ
\tw{\top}\iota_j
\end{equation*}
%%%%
As $\tw{\top}\iota$ and $\tw{\top}\iota^{\ast}$ factor naturally over the direct
product, we see immediately that (7) holds.

Now we show that this is the only family of maps satisfying these properties.
For this, we assume we have a family of bijections $f_s^{\bG} : \cE(G,s) \to
\cE(\rC_{G^{\ast}}(s),1)$ satisfying the conditions of Theorem~\ref{UniqueJord}. 
Fix $(\bG,F)$ and $s$ and let $\iota : (\bG,F) \to (\wt\bG,F)$ be a regular
embedding with dual $\iota^{\ast} : (\wt\bG^{\ast},F^{\ast}) \to (\bG,F)$.

Let $\tilde{s} \in \wt G^{\ast}$ be an element such that
$\iota^{\ast}(\tilde{s}) = s$. Within our family of bijections we have a
bijection $f_{\tilde{s}}^{\wt\bG} : \cE(\wt G,\tilde{s}) \to \cE(\rC_{\wt
G^{\ast}}(\tilde{s}),1)$. Since $\iota$ is an isotypy, it follows from (6) of
Theorem~\ref{UniqueJord}, that
%%%%
\begin{equation} \label{final}
f_{\tilde{s}}^{\wt\bG} = \tw{\top}\iota^{\ast} \circ f_s^{\bG} \circ \tw{\top}\iota
:
\cE(\wt G,\tilde s) \to \cE(\rC_{\wt G^{\ast}}(\tilde s), 1),
\end{equation}
%%%%
where $\tilde s \in \wt G^{\ast}$ is an element satisfying $s =
\iota^{\ast}(\tilde s)$.  On the other hand, by Lemma~\ref{lem:Js-is-DM-bij}, we
know that $f_{\tilde{s}}^{\wt\bG} = J_{\tilde{s}}^{\wt\bG}$ must be
Digne--Michel's unique Jordan decomposition.  Now by Lemma \ref{restrictionbij}
the unique bijection $f_s^{\bG}$ which can satisfy (\ref{final}) is $J_{s,
\iota}^{\cD} = J_s^{\bG}$ as in (\ref{eq:def-Js}).  Thus $f_s^{\bG} =
J_s^{\bG}$.
\end{proof}

\end{document}